\documentclass[leqno,12pt]{amsart}

\usepackage{palatino}
\usepackage[mathcal]{euler}

\usepackage{amssymb,amsmath,amsthm,enumerate,slashed}
\usepackage[latin1]{inputenc}
\usepackage[T1]{fontenc}
\usepackage{mathrsfs}
\usepackage[dvipsnames,svgnames,table]{xcolor}
\usepackage{comment}

\definecolor{DarkBlue}{rgb}{0.2,0.2,0.4}
\usepackage[colorlinks=true, linkcolor=DarkBlue, urlcolor= DarkBlue, citecolor=DarkBlue]{hyperref}

\usepackage{graphicx}
\usepackage{epstopdf,epsfig}

\usepackage{xy}	
\xyoption{all}
\usepackage{marginnote}
\usepackage[mmddyyyy]{datetime}

\usepackage[normalem]{ulem}

\makeatletter
\def\l@subsection{\@tocline{2}{0pt}{2.5pc}{5pc}{}}
\makeatother

\newcommand{\lie}[1]{\mathfrak{#1}}
\newcommand{\compop}{\mathfrak{K}}

\newcommand{\Compacts}{\mathfrak{K}}

\newcommand{\R}{\mathbb{R}}
\newcommand{\C}{\mathbb{C}}

\DeclareMathOperator{\Ind}{Ind}

\DeclareMathOperator{\Ad}{Ad}

\newcommand{\InfCh}{\operatorname{Inf\,Ch}}

\newcommand{\ImInfCh}{\operatorname{Im\,Inf\,Ch}}

\newcommand{\Gtemp}{\widehat{G}_{\mathrm{tempered}}}
\newcommand{\bigG}{\pmb{G}}
\newcommand{\bigN}{\mathbf N}

\swapnumbers
\numberwithin{equation}{subsection}

\theoremstyle{plain}
\newtheorem*{theorem*}{Theorem}
\newtheorem*{lemma*}{Lemma}

\newtheorem{theorem}[equation]{Theorem}
\newtheorem{lemma}[equation]{Lemma}

\newtheorem{corollary}[equation]{Corollary}
\theoremstyle{definition}
\newtheorem{definition}[equation]{Definition}
\newtheorem*{definition*}{Definition}

\newtheorem*{notation*}{Notation}
\newtheorem{remark}[equation]{Remark}
\newtheorem{remarks}[equation]{Remarks}

\newtheorem*{remark*}{Remark}

\def\clap#1{\hbox to 0pt{\hss#1\hss}}

\usepackage{cancel} 

\title[Mackey Embedding for Reduced Group C*-Algebras]{A Mackey Embedding for Reduced C*-Algebras of Real Reductive Groups}
\author{Pierre Clare \and Nigel Higson \and Angel
Rom\'an}
\address{P. Clare, College of William \& Mary,  Williamsburg, VA, USA}
\address{N. Higson, Penn State University, University Park, PA, USA}
\address{A. Rom\'an, Rochester Institute of Technology,  Rochester, NY, USA}

\date{\today}

\begin{document}

\begin{abstract}
The purpose of this paper is to construct an embedding of the $C^*$-algebra of the Cartan motion group of a real reductive group $G$ into the reduced C*-algebra of   $G$ itself.  The embedding has a number of applications: we shall use it to characterize the Mackey bijection from the tempered dual of $G$ into the unitary dual of the motion group; to characterize the continuous field of reduced group $C^*$-algebras arising from the contraction of $G$ to its Cartan motion group; and to characterize the Connes-Kasparov assembly map in operator $K$-theory. Our results continue and complete a project that was begun several years ago by the last two authors, who considered the case of complex groups.  In the real case, detailed information from the theory of $R$-groups is used in the construction. 
\end{abstract}

\maketitle
 
\section{Introduction}

Let $G$ be a real reductive group and let $K$ be a maximal compact subgroup of $G$.  The \emph{Cartan motion group} associated to $G$ and $K$ is the semidirect product Lie group $G_0{=}K{\ltimes }{\mathfrak{g}/\mathfrak{k}}$.   It is a  first-order approximation to $G$, near $K$, but although it is designed to resemble $G$, its structure is quite different.  Nevertheless,  George Mackey  proposed in \cite{Mackey75} that the  irreducible unitary representations of the  $G_0$   ought to  ``correspond'' to those of the group $G$.  Mackey's proposal initially failed to attract much attention, but thirty years later it was  explored and made precise by the second author for complex semisimple groups \cite{Higson08}, using  the  \emph{tempered} representations of $G$. Quite recently Alexandre Afgoustidis \cite{AfgoustidisMackeyBijection} constructed a  \emph{Mackey bijection} from the tempered dual of \emph{any} real reductive group to the unitary dual of its Cartan motion group.

The purpose of this paper is to construct a \emph{Mackey embedding} of group $C^*$-algebras
\[
\alpha \colon C^*(G_0)\longrightarrow C^*_r(G) ,
\]
that further develops Mackey's idea of a correspondence between the representation theories of the groups $G_0$ and $G$, and that generalizes the construction of a Mackey embedding for complex reductive groups  in \cite{HigsonRoman20}. We shall show  that every irreducible tempered unitary representation $\pi$ of a real reductive group $G$, when viewed as a representation of the reduced group $C^*$-algebra of $G$, pulls back \textit{via} $\alpha$ to a unitary representation of $G_0$ that includes, as a subrepresentation, the irreducible representation of $G_0$ which corresponds to $\pi$ under the Mackey  bijection of Afgoustidis. In fact we shall show that  the \emph{only} bijection with this property is the one that Afgoustidis constructed.

Our Mackey embedding is constructed as follows. First, the groups $G_0$ and $G$ fit into a smooth, one-parameter family of Lie groups $\{ G_t \}_{t\in \R}$ in which $G_t=G $ whenever $t\ne 0$.  The  family may be constructed using the deformation to the normal cone from algebraic geometry \cite{Fulton84}; see Section~\ref{sec-cts-fields}.
Next, we may form  from this smooth family of Lie groups a continuous field of group $C^*$-algebras $\{ C^*_r (G_t)\}_{t\in \R}$.

Now, there is  an interesting  connection between the parameter $t$ in the continuous field $\{ C^*_r (G_t)\}_{t\in \R}$  and a \emph{rescaling} action of the positive real numbers on the tempered dual of $G$. Roughly speaking, the latter is  just the operation 
\[
t\cdot \pi_{\sigma, \nu}= \pi_{\sigma, t^{-1}\nu},
\]
where $t$ is a positive real number, and where $\sigma$ and $\nu$ are the discrete and continuous parameters describing a tempered representation of $G$; see Section~\ref{sec-rescaling-maps-on-tempered-dual}.  We shall show that the rescaling action  on the tempered dual is related to the continuous field by the \emph{limit formula} 
\[
\lim_{t\to \infty} \pi_{\sigma, t^{-1}\nu} (f_t) = \rho_{\sigma,\nu}(f_0),
\]
in which $\{f_t\}_{t\in \R}$ is any continuous section of $\{ C^*_r (G_t)\}_{t\in \R}$ and $\rho_{\sigma,\nu}$ is the natural unitary (but not necessarily  irreducible) representation of the motion group $G_0$ that is constructed from the same parameters $\sigma$ and $\nu$. See Section~\ref{sec-limit-formula}
for details.

In addition to the limit formula, our construction of the Mackey embedding involves a lift of the re\-scaling action on the tempered dual of $G$ to an action by automorphisms on the reduced $C^*$-algebra of $G$,
\[
\alpha_t \colon C^*_r (G) \longrightarrow C^*_r (G) \qquad (t>0).
\]
This is the most difficult part of our construction, and it is here where the gap between the complex case, treated in \cite{HigsonRoman20}, and the real case is the widest.   The rescaling action on the tempered dual of a real reductive group is relatively easy to define, despite the  complicated form of the tempered dual in the real case.  But the existence of non-trivial self-intertwining operators on general cuspidal principal series representations in the real case makes the problem of lifting  the rescaling action to the reduced group $C^*$-algebra much more difficult than it is in the complex case.  Our construction involves a study of the $R$-group, and is carried out in Section~\ref{sec-rescaling-automorphisms}.

Having constructed the rescaling automorphisms, it is a simple matter to define the Mackey embedding by the formula\footnote{The formula presented here ignores a small issue involving choices of Haar measure on the copies $G_t$ of the group $G$.  See Section~\ref{sec-bijection-characterization} for details.}
\[
\alpha(f_0) = \lim_{t\to 0} \alpha_t (f_t),
\]
where $\{ f_t\}$ is any continuous section of the continuous field.  The limit exists thanks to the limit formula.

We shall explain in Section~\ref{sec-bijection-characterization} how to characterize  the continuous field $\{C^*_r (G_t)\}_{t\in \R}$ in terms of the Mackey embedding: it is the so-called \emph{mapping cone field} associated to the embedding. We shall also prove that the Connes-Kasparov isomorphism in operator $K$-theory is equivalent to the assertion that our Mackey embedding induces an isomorphism in $K$-theory.  Finally,   we shall  give the characterization of the Mackey bijection that we have already described above.

\smallskip

\noindent{\bf Acknowledgments.}
The authors thank the referee for many very helpful observations.  
Nigel Higson's research was supported by the NSF grant DMS-1952669, and Angel Rom\'an's research was partially supp\-orted by the NSF grant DMS-2213097. All the authors gratefully acknowledge support of the Institut Henri Poincar\'e (UAR 839 CNRS-Sorbonne Universit\'e), and LabEx CARMIN (ANR-10-LABX-59-01).

\section{Deformation spaces and continuous fields}
\label{sec-cts-fields}

In this section we shall record  some   facts about the deformation to the normal cone construction in geometry, and about the associated continuous field of $C^*$-algebras in the group case. Essentially the same background material appears in  \cite{HigsonRoman20}, so we shall be brief.

\subsection{Deformation to the normal cone}
\label{sec-dnc}
Let $M\longrightarrow V$ be a closed embedding   of  smooth manifolds. The associated \emph{deformation to the normal cone}  $\bigN_VM$  is a smooth manifold that is equipped with a submersion 
\begin{equation}
    \label{eq-submersion-from-def-to-normal-cone-to-r}
\bigN_VM \longrightarrow \R.
\end{equation}
The fiber of this submersion  over any  $t\ne 0$ is a copy of $V$, while the fiber over $t=0$ is  a copy of the normal bundle 
\[
N_V M = TV\vert _M \,\big / \,TM.
\]
So as a set, $\bigN _V M$ is the disjoint union  
\[
\bigN _VM =N_V M {\times} \{0\} \, \, \sqcup\,\,  \bigsqcup_{t\ne 0}  V{\times} \{ t\} .
\]
But we equip the disjoint union with the unique smooth manifold structure for which
\begin{enumerate}[\rm (i)]

\item the  map $\bigN _VM \to \R$ in \eqref{eq-submersion-from-def-to-normal-cone-to-r}  is a smooth submersion,

\item if $f$ is a smooth function  on $V$, then the function 
\[
\begin{cases} (v,t) \longmapsto f(v) & t \ne 0 \\
(X_m,0)\longmapsto f(m)
\end{cases}
\]
is smooth on $\bigN _V M$,   

\item if $f$ is a smooth function  on $V$,  and if $f$ vanishes on $M$, then the function 
\[
\begin{cases} (v,t) \longmapsto t^{-1} f(v) & t \ne 0 \\
(X_m,0)\longmapsto X_m(f)
\end{cases}
\]
is smooth on $\bigN _V M$, and 

\item at every point of $\bigN _VM$, local coordinates can be selected from functions of the above types.

\end{enumerate} 

The deformation to the normal cone construction is a functor from closed embeddings  $M \longrightarrow V$ to smooth manifolds over $\R$, and it follows from functoriality that in the case of a closed embedding of Lie groups  $K{\to}G$,  the space 
\begin{equation}
    \label{eq-def-of-big-g}
\bigG = \bigN _GK = \{ G_t \} _{t \in \R}
\end{equation}
is a smooth family of Lie groups over $\R$. If we  trivialize the tangent bundles on $G$ and $K$ by left translations, then there is an induced trivialization
\begin{equation*}
N_G K \cong K \times \lie{g}/ \lie {k} ,
\end{equation*}
and the fibers of $\bigG $  are the groups
\begin{equation}
    \label{eq-fiber-of-the-smooth-family}
G_t = \begin{cases} 
G & t\ne 0 
\\
 K \ltimes \lie{g}/ \lie {k} & t =0.
 \end{cases}
\end{equation}

\subsection{The associated continuous field of C*-algebras}
\label{sec-continuous-field-of-group-c-star-algebras}
From now on  we shall assume that  $G$ is an almost-connected Lie group (that is, a Lie group with finitely many path components) and that $K$ is a maximal compact subgroup.  
Fix a left Haar measure $\mu$  for $G$,  determined by some nonzero element of the highest exterior power $\Lambda^{\mathrm{top}} \mathfrak{g}^*$ and a compatible orientation of $G$.
 Then for $t{\neq} 0$ equip the fiber groups    $G_t$ in \eqref{eq-fiber-of-the-smooth-family} 
with the left Haar measures   
\begin{equation}
\label{eq-def-of-d-and-haar-measures}
 dg_t = |t|^{-\dim (G/K)}dg .
\end{equation}
These  Haar measures vary smoothly with $t$, and they extend smoothly to $t{=}0$, in the sense that if $\xi$ is a smooth and compactly supported function on $\bigG$, and if $\xi_t$ denotes the restriction of $\xi$ to $G_t$, then  
\[
t \longmapsto \int_{G_t} \xi_t(g_t) \; dg_t  
\]
is a smooth function of $t{\in} \R$.  The measure on $G_0$ is determined  by  the natural identification $\Lambda ^{\mathrm{top}} ( \mathfrak{k}^* {\times} (\mathfrak{g}/\mathfrak{k})^*) \cong  \Lambda ^{\mathrm{top}} \mathfrak{g}^*$.

 If  $\xi_1$ and $\xi_2$ are two smooth and compactly supported functions on $\bigG$, then their convolution product, defined fiberwise by 
\[
(\xi_{1,t}*\xi_{2,t}) (g_t) =  \int _{G_t } \xi_{1,t} (h_t ) \xi_{2,t}(h_t ^{-1}g_t)\; d  h_t \qquad (g_t\in G_t),
\]
is a smooth and compactly supported function $\bigG$, too (here $\xi_{1,t}$ and $\xi_{2,t}$ denote the restrictions of $\xi_1$ and $\xi_2$ to $G_t$).

Of course, the same convolution formula, when applied to functions on $G_t$ alone, defines a product on $C_c^\infty(G_t)$, and a representation of $C_c^\infty (G_t)$ as bounded operators on $L^2 (G_t,dg_t)$ by left-convolution.  As usual, we shall denote by $C^*_r(G_t)$ the   completion of $C_c^\infty (G_t)$ in the operator norm on $L^2 (G_t,dg_t)$.    Compare  \cite[Sec.\;7.2]{Pedersen79}.

\begin{lemma}[See {\cite[Lemma 6.13]{Higson08}}] 
If $\xi\in C_c^\infty (\bigG )$,   and if $\xi_t$ denotes the restriction of $\xi$ to $G_t$, then the norm
 $\|\xi_t\|_{C^*_r (G_t)}$ 
is a continuous function of $t\in \R$.
\end{lemma}

The smooth and compactly supported functions on $\bigG$ therefore generate the continuous sections of a \emph{continuous field of $C^*$-algebras} over $\R$,  in the sense of \cite[Prop.\;10.2.3]{Dixmier77}, with fibers $C^*_r (G_t)$, in which  a section $\{\xi'_t\}$  is deemed to be continuous if and only if for every $t_0\in \R$ and every $\varepsilon {>} 0$, there is a neighborhood $U$ of $t_0\in \R$ and a smooth, compactly supported   $\xi\colon \bigG\to\C$ such that $\| \xi_t - \xi'_t \|_{C^*_r(G_t)} <\varepsilon$, for all $t\in U$.

\subsection{Mapping cones and rescaling automorphisms}
One of the goals of this paper is to determine the structure of the continuous field $\{ C_r^*(G_t)\}$ in the case of a real reductive group in terms of the following construction:
 
 \begin{definition}
 \label{def-mapping-cone-field}
 Let $\alpha:A_0 \longrightarrow A$ be an embedding of  $C^*$-algebras. The \emph{mapping cone continuous field} of $C^*$-algebras over $\R$ associated to $\alpha$ has fibers 
\[
\operatorname{Cone}(\alpha)_{t}=\begin{cases}
A & t\neq 0\\
A_0 & t=0.
\end{cases}
\]
Its continuous sections are all those set-theoretic sections $\{ a_t\}$ for which the function 
$$
t\longmapsto  \begin{cases}
a_t & t\neq 0\\
\alpha (a_0) & t=0 
\end{cases}
$$
from $\R$ to $A$ is norm-continuous. 
\end{definition}

As advertised in the introduction, for  a real reductive  group $G$ and maximal compact subgroup $K$, we shall eventually construct an embedding of $C^*$-algebras 
\[
\alpha \colon C^*_r (G_0) \longrightarrow C^*_r (G),
\]
which we shall call the \emph{Mackey embedding}, for which the associated mapping cone continuous field is isomorphic to the continuous field $\{ C^*_r (G_t)\}$.  The isomorphism of continuous fields will be obtained using the following simple observation, which we shall state for continuous fields over the half-line $[0,\infty)$, rather than over $\R$, in order to simplify matters later on. 

\begin{lemma}
\label{lem-embedding-from-limit-formula-and-mapping-cone-field}
Let $\{ A_t\}$ be a continuous field of $C^*$-algebras over the parameter space $[0,\infty)$.  Suppose that there is a continuous family of $C^*$-algebra isomorphisms
\[
\alpha_t \colon A_t \longrightarrow A \qquad (t > 0)
\]
with the property that for every continuous section $\{ a_t\}$ of $\{ A_t\}$ over $[0,\infty)$, the limit 
\[
\lim_{t\to 0} \alpha_t (a_t)
\]
exists in $A$.  Then the formula 
\[
\alpha (a_0) = \lim_{t\to 0} \alpha_t (a_t) ,
\]
where $\{ a_t\}$ is any continuous section extending $a_0\in A_0$, determines an embedding of $C^*$-algebras
\[
\alpha\colon A_0 \longrightarrow A.
\]
Moreover  the isomorphisms $\alpha_t$, along with $\alpha_0{=}\mathrm{id}$, determine an isomorphism from the continuous field $\{ A_t\}$ to the mapping cone field for the morphism $\alpha \colon A_0 \to A$. \qed
\end{lemma}

In Section~\ref{sec-rescaling-automorphisms}    we shall construct isomorphisms $ \alpha_t\colon C^*_r(G_t)\to C^*_r(G)$ to which the lemma will  apply (although we shall use a slightly different notation there). The existence of the limits  in the lemma will be proved in Section \ref{sec-limit-formula}. 

\section{Parabolic induction and intertwining operators}

In this section we shall recall some facts about parabolically induced representations and intertwining operators between them. In addition to fixing terminology and notation, we shall also prove one technical theorem about families of intertwining operators that is a consequence of the principle of induction in stages. 

\subsection{Real reductive groups and representations}

We shall be working throughout the paper  with  {real reductive groups} in the sense of David Vogan's definitions in \cite[Sec.\,0.1]{VoganGreenBook}. In particular our real reductive groups will be linear. Fraktur letters without subscript will denote the \emph{real} Lie algebras of the associated Lie groups.  

Let $G$ be a real reductive group. As part of the definition, $G$ is equipped with  a Cartan involution $\theta$, and we shall denote by $K$  the maximal compact subgroup of $G$ that is fixed by $\theta$. The corresponding Cartan decomposition of the Lie algebra $\mathfrak{g}$ will be written as \begin{equation}
\label{eq-Cartan-decomp}
\mathfrak{g}=\mathfrak{k}\oplus\mathfrak{s}.
\end{equation}  
By definition, the Lie algebra $\mathfrak{g}$ is also equipped  with a non-degen\-erate, $G$-invariant, symmetric bilinear form 
\begin{equation}\label{eq-Cartan-form}
\left\langle\:,\:\right\rangle:\mathfrak{g}\times\mathfrak{g}\longrightarrow \R
\end{equation} 
that is compatible with the Cartan decomposition \eqref{eq-Cartan-decomp} in the sense that it is positive-definite on $\mathfrak{s}$ and negative-definite on $\mathfrak{k}$. We shall also fix throughout the paper a maximal abelian subspace $\mathfrak{a}\subseteq\mathfrak{s}$ and a compatible Iwasawa decomposition 
\begin{equation}\label{eq-fixed-iwasawa-decomp}
G=KAN.  
\end{equation}
The group $N$ is a closed connected subgroup of $G$, and its Lie algebra has the form 
\begin{equation}
    \label{eq-lie-algebra-of-n}
\mathfrak{n} = \bigoplus _{\alpha \in \Delta^+(\mathfrak{g},\mathfrak{a})} \mathfrak{g}_\alpha,
\end{equation}
where the direct sum is over a system of positive restricted roots. See \cite[Sec. VI.4]{Knapp02}.

We shall work throughout with strongly continuous and unitary representations on Hilbert spaces. The notation $H_\pi$ will refer to the  carrying Hilbert space of the representation $\pi$.   With the obvious exception of the regular representation on $L^2(G)$, we shall also assume that our representations are \emph{admissible}, which in the context of unitary representations means that they are finite direct sums of irreducible unitary representations.

\subsection{Parabolic subgroups}
Given $G$ and the attendant choices made in the previous section, form the \emph{minimal parabolic subgroup} $P_{\min}=MAN$, where $M$ is the centralizer of $\mathfrak{a}$ in $K$. The \emph{standard parabolic subgroups} of $G$ are the closed subgroups $P\subseteq G$ that contain $P_{\min}$. There are finitely many of these, each of which may be written as
\[
P = L_P\cdot N_P=M_P\cdot A_P\cdot N_P,
\]
(these are direct product decompositions as manifolds) 
where: 
\begin{enumerate}[\rm (i)]

\item $L_P$ is the  \emph{Levi factor} of $P$, which is a real reductive group in its own right, and $N_P$ is the \emph{unipotent radical} of $P$ and a closed, connected subgroup of $N$. 

\item  $M_P$ is the smallest closed subgroup of $L_P$ that includes all the compact subgroups of $L_P$ (equivalently, $M_P$ is  the mutual kernel of all continuous morphisms from $L_P$ to $\R^\times _{+}$) and $A_P$ is the largest subgroup of  $A$ that centralizes $L_P$; one has  $L_P = M_P\cdot A_P$.

\end{enumerate}
Using the notation of \eqref{eq-lie-algebra-of-n}, the Lie algebra of $N_P$ is 
\begin{equation}
    \label{eq-lie-algebra-of-n-p}
\mathfrak{n}_P = \bigoplus_{\substack{\alpha \in \Delta^+(\mathfrak{g},\mathfrak{a}) \\ \alpha\vert_{\mathfrak{a}_P \ne 0}}} \mathfrak{g}_\alpha. 
\end{equation}
See \cite[Sec.~V.5]{KnappRepTheorySemisimpleGroups} or \cite[Sec.~VII.7]{Knapp02}. To be consistent throughout the paper we shall always work with standard parabolic subgroups (the other parabolic subgroups of $G$ are  conjugates of the standard parabolic subgroups).  
Note that the group $G$ itself is a standard parabolic subgroup, with $G= M_G A_G$ (the group $N_G$ is trivial). This is the \emph{split decomposition} of $G$ \cite[Prop.\,7.27]{Knapp02}.

\subsection{Parabolic induction}
\label{sec-parabolic-induction}
Let $P=L_PN_P$ be a standard parabolic subgroup of $G$, as above. Denote by $\delta_P$ the modular function of $P$, which is defined
by the formula
\begin{equation}\label{eq-def-of-delta-p}
\int_P \xi(pp') \, dp = \delta_P(p') \int_P \xi(p)\, dp,
\end{equation} 
where $dp$ denotes a left Haar measure on $P$, and let  $\pi$ be an  
admissible  unitary representation of $L_P$. Denote by  $H_\pi^\infty$  the space of smooth vectors for $\pi$, and then form the Fr\'echet space 

\begin{equation}
    \label{eq-smooth-parabolic-representation-space}
\Ind_P^{G,\infty} H^\infty_\pi
=
\Bigl \{ \, 
\varphi: G \stackrel{C^\infty}\to H^\infty _\pi \Big \vert  
\parbox{2.1 in}{\begin{center}$\varphi(g\ell n ) = \pi(\ell)^{-1}\delta_P(\ell)^{-\frac 1 2} \varphi(g)$
\\
$\forall g\in G, \,\, \forall \ell\in L_P,\,\,  \forall n\in N_P$\end{center} }
\,\Bigr\} .
\end{equation}
Define an inner product on this space by 
\begin{equation}
    \label{eq-unitary-inner-product}
\langle \varphi_1,\varphi_2 \rangle_{\Ind_P^G H_\pi} = \frac{1}{\operatorname{vol} (K)} \int _K \langle \varphi_1(k), \varphi_2(k)\rangle_{H_\pi}\, dk ,
\end{equation}
and write 
\begin{equation}
    \label{eq-parabolic-representation-hilbert-space}
 \Ind_P^G H_\pi = \parbox{2.8 in}{completion of  $\Ind_P^{G,\infty} H^\infty_\pi$ in the norm associated to the inner product \eqref{eq-unitary-inner-product}.}
\end{equation}
The group $G$ acts by left translation on $\Ind_P^{G,\infty} H^\infty_\pi$, and the action preserves the inner product \eqref{eq-unitary-inner-product}, thanks to the inclusion of the $\delta_P$-term in \eqref{eq-smooth-parabolic-representation-space}. So the action extends to a unitary action on $\Ind_P^G H_\pi$, which becomes a unitary representation of $G$, denoted by $\Ind_P^G\pi$. This is  the representation \emph{parabolically and unitarily induced from the representation $\pi$ of $L_P$ along $P$}, and the space $\Ind_P^{G,\infty} H^\infty_\pi$ is the space of smooth vectors in $\Ind_P^G H_\pi$. More details can be found for instance in \cite[\S VII.1]{KnappRepTheorySemisimpleGroups}.

The parabolic induction construction \begin{equation}\label{eq-parabolic-induction}
H_\pi \longmapsto \Ind_P^G H_\pi
\end{equation}  is a functor from admissible unitary representations of $L$ to admissible unitary representations of $G$.

\subsection{Compact model}\label{sec-compact-model}
The Fr\'echet space $\Ind_P^{G,\infty} H^\infty_\pi$ may be identified with the space of smooth functions 
\[
C^\infty (K, H_\pi^\infty )^{K \cap L_P} =  \Bigl\{\varphi: K \overset{C^\infty}{\longrightarrow} H^\infty _\pi\Big \vert \parbox{1.6in}{
\begin{center}
$\varphi(k\ell) = \pi(\ell)^{-1}\varphi(k)$
\\
$\forall k\in K\:,\:\forall \ell\in K\cap L_P$ 
\end{center}
}
\Bigr\} 
\]
\textit{via} restriction  of functions from $G$ to $K$. The restriction map is  an isomorphism of Fr\'echet spaces, and it
 extends to a unitary isomorphism of Hilbert spaces 
\begin{equation}
    \label{eq-compact-model-Hilbert-space-isomorphism}
\Ind_P^G H_\pi \stackrel \cong \longrightarrow  L^2 (K, H_\pi)^{K \cap L_P}
\end{equation}
after one completes on both sides with respect to the inner product in \eqref{eq-unitary-inner-product}.  

\begin{definition} 
The \emph{compact model} of the parabolically induced representation $\Ind_P^G\pi$ defined in \eqref{eq-parabolic-induction} is the unitary representation of $G$ on $L^2 (K, H_\pi)^{K \cap L_P}$ that is obtained
using the  isomorphism \eqref{eq-compact-model-Hilbert-space-isomorphism}.
\end{definition} 

The action of $G$ on the compact model Hilbert space is a bit complicated to describe, but its restriction to $K$ is simply left-translation. 

We shall use the following simple facts about the compact model and the split decomposition of $G$, which follow directly from the definitions: 

\begin{lemma} 
\label{lem-parabolically-induced-reps-on-the-split-component}
Let $G$ be a real reductive group, let $P=L_PN_P$ be a standard parabolic subgroup, and let $\pi$ be an 
admissible unitary representation of $L_P$ on a Hilbert space $H_\pi$.
Let  $G= M_GA_G$ be the split decomposition of $G$ .
\begin{enumerate}[\rm (i)]

\item  The compact model Hilbert space $L^2(K,H_\pi)^{K\cap L_P}$ depends only on the restriction of $\pi$ to $M_G\cap L_P$.
\item The restriction to $M_G$ of the representation $\Ind_P^G \pi$ on the compact model Hilbert space depends only on the restriction of $\pi$ to $M_G\cap L_P$. That is, if $\pi_1,\pi_2:L_P\longrightarrow U(H_\pi)$ have a common restriction to $M_G\cap L_P$, then the restriction to $M_G$ of the associated parabolically induced representations on the compact model Hilbert spaces are \emph{identical}.
\item The restriction to $A_G$ of the representation $\Ind_P^G \pi$ on the compact model Hilbert space is given by the formula
\[
\pushQED{\qed}
(a\cdot \varphi)(k) = \pi(a)\cdot \varphi(k)\qquad \forall a\in A_G,\,\, \forall k \in K .
\qedhere
\popQED
\]

\end{enumerate}
\end{lemma}

\subsection{Induction in stages}\label{sec-induction-in-stages}

Suppose  that $P$ and $Q$ are standard  parabolic subgroups of $G$ with  
$P\subseteq Q \subseteq G $. 
There are  inclusions 
\[
L_P\subseteq L_Q\subseteq G,
\quad A_Q\subseteq A_P 
\quad\text{and} \quad 
M_P\subseteq M_Q,
\] 
and moreover  the subgroup $P{\cap}L_Q\subseteq L_Q$ is a parabolic subgroup of $L_Q$, with 
\[
P{\cap}L_Q = L_P\cdot (N_P{\cap} L_Q)= M_P\cdot A_P \cdot (N_P{\cap} L_Q).
\]
See for instance \cite[Lemma 4.1.17]{VoganGreenBook}. 
In addition,  $Q=L_QP$. Finally,  $N_P$ admits the semidirect product decomposition 
\[
N_P = (L_Q{\cap} N_P) \ltimes N_Q.
\]

Given an  
admissible unitary representation $\pi$ of $L_P$, we may form the Hilbert space $\Ind_{P{\cap}L_Q}^{L_Q,\infty} H^\infty_\pi$, which carries an  
admissible   representation  of $L_Q$, and then the space 
\[ 
\Ind_{Q}^{G,\infty} \Ind_{P{\cap}L_Q}^{L_Q,\infty} H^\infty_\pi ,
\]
which carries an  
admissible representation of $G$.
The  elements of the latter are functions $F$ on $G$ with values in the space $\Ind_{P{\cap}L_Q}^{L_Q,\infty} H^\infty_\pi$ of $H^\infty_\pi$-valued functions on $L_Q$.  Given such a function $F$, and given $g\in G$, we may evaluate $F(g)$ at $e\in L_Q$ to obtain an element 
$F(g)(e) \in H_\pi ^\infty$. 
This process creates from $F$ a smooth function on $G$ with values in $H^\infty _\pi$: 
\begin{equation}\label{eq-induction-in-stages-formula}
\Ind_{Q}^{G,\infty} \Ind_{P{\cap}L_Q}^{L_Q,\infty} H^\infty_\pi \ni F \longmapsto  \bigl [ g\mapsto F(g)(e)\bigr ] \in  \Ind_P^{G,\infty} H^\infty _\pi 
\end{equation}

\begin{lemma}[Induction in stages; see e.g.\ {\cite[Prop.\,4.1.18]{VoganGreenBook}}]
\label{lem-induction-in-stages}
The formula \eqref{eq-induction-in-stages-formula} defines a $G$-equivariant and isometric linear isomorphism that extends to a unitary equivalence of representations 
\[
\Ind_{Q}^G \Ind_{P{\cap}L_Q}^{L_Q} H_\pi \stackrel \cong\longrightarrow \Ind_P^G H_\pi
\]
and an equivalence of functors from admissible unitary representations of $L_P$ to  
admissible unitary representations of $G$.
\end{lemma}

The following lemma examines induction in stages from the point of view of the compact model. 

\begin{lemma}
\label{lem-induction-in-stages-in-compact-model}
    The  $G$-equivariant and isometric linear map \eqref{eq-induction-in-stages-formula} fits into a functorial \textup{(}in $\pi$\textup{)} commuting diagram 
\begin{equation*}
\xymatrix{
    \Ind_{Q}^{G,\infty} \Ind_{P{\cap}L_Q}^{L_Q,\infty} H^\infty_\pi\ar[d]_{\cong} \ar[r]^-{\eqref{eq-induction-in-stages-formula}} & \Ind_P^{G,\infty} H^\infty_\pi \ar[d]^{\cong}
    \\
    C^\infty\bigl (K, C^\infty(K{\cap}L_Q, H^\infty_\pi )^{K \cap L_P}\bigr)^{K\cap L_Q} \ar[r] & C^\infty \bigl (K, H^\infty _\pi\bigr )^{K\cap L_P}
}
\end{equation*}
in which the vertical maps are the restriction isomorphisms as in \eqref{eq-compact-model-Hilbert-space-isomorphism}, and the bottom morphism is defined using the  formula 
\begin{multline*} 
C^\infty\bigl (K, C^\infty(K{\cap}L_Q, H^\infty_\pi )^{K \cap L_P}\bigr)^{K\cap L_Q} \ni F 
\\
\longmapsto  \bigl [ k\mapsto F(k)(e)\bigr ] \in C^\infty \bigl (K, H^\infty _\pi\bigr )^{K\cap L_P} ,
\end{multline*}
as in  \eqref{eq-induction-in-stages-formula}.  The bottom morphism depends only on the restriction of $\pi$ to $K\cap L_P$, and, like the top morphism, it extends to an isometric isomorphism between Hilbert space completions. \qed
\end{lemma}

\subsection{Intertwining operators in the compact model}

We shall require   some parts  of the theory of intertwining operators between parabolically induced representations, as developed by Knapp and Stein \cite{KnappSteinI,KnappSteinII}. But we shall also need a more specialized theorem about intertwining operators that is  difficult to locate in the literature in the precise form that we require.  The purpose of this section is to formulate and prove that result.

Let $P=L_PN_P$ be a standard parabolic subgroup of $G$, and let $\pi$ be an 
admissible unitary representation of $L_P=M_PA_P$ on a Hilbert space $H_\pi$.  For $\nu\in \mathfrak{a}^*_P$,  
denote by 
\begin{equation}
\label{eq-notation-e-to-the-i-nu-times-tau-1}
e^{i\nu}\cdot \pi  \colon L_P \to U (H_\pi)
\end{equation}
the unitary representation that is defined by the formula 
\begin{equation*}
(e^{i\nu}\cdot \pi)(\ell) = e^{i\nu}(a)\cdot  \pi(\ell)\qquad \forall \ell=ma \in L_P
\end{equation*}
(this is not  standard notation, but it is convenient for our  purposes).

Note that  the carrying Hilbert space of the representation $e^{i\nu}\cdot \pi$ is $H_\pi$, and indeed the underlying $(K{\cap}L_P)$-representation is $H_\pi$, too. 
It follows that under the restriction isomorphism 
\[
\xymatrix{
\Ind_P^G  H_{e^{i\nu}\cdot\pi}\ar[r]^-{\eqref{eq-compact-model-Hilbert-space-isomorphism}}_-\cong &  L^2 (K, H_{e^{i\nu} \cdot\pi})^{K\cap L_P} ,
}
\]
the Hilbert space on the right-hand side does not depend on $\nu$ in any way.  To emphasize  this, we shall write 
\[
\xymatrix{
\Ind_P^G  H_{e^{i\nu}\cdot \pi}\ar[r]^-{\eqref{eq-compact-model-Hilbert-space-isomorphism}}_-\cong &  L^2 (K, H_{\pi})^{K\cap L_P} 
}
\]
from now on.

We shall construct  intertwining operators between parabolically induced representations using copies of  $\mathrm{SL}(2,\R)$ or $\mathrm{PSL}(2,\R)$ in  $G$. Let  $P$ be a standard parabolic subgroup,  and let $\varphi:\mathfrak{sl}(2,\R)\to \mathfrak{g}$ be a Lie algebra homomorphism such that 
\[
\varphi\left(\left[\begin{smallmatrix}
0&1\\-1&0\end{smallmatrix}\right]\right)\in\mathfrak{k}\quad,\quad \varphi\left(\left[\begin{smallmatrix}
1&0\\0&-1\end{smallmatrix}\right]\right)\in\mathfrak{a}_P\quad\text{and}\quad\varphi\left(\left[\begin{smallmatrix}
0&1\\0&0\end{smallmatrix}\right]\right)\in\mathfrak{n}_P
\]
(compare \cite[4.3.6]{VoganGreenBook}). The map $\varphi$ exponentiates to a Lie group morphism $\Phi\colon \mathrm{SL}(2,\R)\to G$, and we shall write 
\begin{equation}
\label{eq-definition-of-x-phi-and-w}
X_\varphi = \varphi \Bigl (
\left [
\begin{smallmatrix}
    1 & 0 \\ 0 & -1
\end{smallmatrix}
\right ]
\Bigr )
\quad \text{and}\quad 
w = \Phi\Bigl (
\left [
\begin{smallmatrix}
    0 & -1 \\ 1 & 0
\end{smallmatrix}
\right ]
\Bigr ).
\end{equation}
Assume the following:  
\begin{equation}
    \label{eq-conditions-on-x-phi}
X \in \mathfrak{a}_P \quad \text{and} \quad X \perp X_\varphi
\quad \Rightarrow \quad 
\text{$\Phi[\mathrm{SL}(2,\R)]$ centralizes $X$}.
\end{equation}
This assumption  will make the intertwining operators that we are about to construct as simple as possible, in the sense that their dependence  on the parameter $\nu\in \mathfrak{a}^*_P$ will be as simple as possible.

The adjoint action of $w$ gives linear automorphisms 
\[
\Ad_{w}\colon \mathfrak{a}_P \longrightarrow \mathfrak{a}_P\quad \text{and} \quad   \Ad^*_{w}\colon \mathfrak{a}_P^* \longrightarrow \mathfrak{a}_P^*, 
\]
the latter being defined by $\Ad^*_{w}(\nu)  = \nu \circ \Ad_{w^{-1}}$. For brevity we shall  write these  as 
\begin{equation}\label{eq-simple-reflection-from-sl2}
w\colon \mathfrak{a}_P \longrightarrow \mathfrak{a}_P\quad \text{and} \quad   w\colon \mathfrak{a}_P^* \longrightarrow \mathfrak{a}_P^* 
\end{equation}
from now on.  The former maps $X_\varphi$ to its negative and is the identity on the orthogonal complement of $X_\varphi$; that is, it is a \emph{simple reflection}.   

Since $w$ normalizes $\mathfrak{a}_P$, it also normalizes the Levi subgroup $L_P$. 
Notice  that if $\pi$ is a unitary representation of $L_P$ on a Hilbert space $H_\pi$, then the Hilbert space of the representaion $w(\pi){=}\Ad_w^*(\pi){=}\pi\circ \Ad_{w^{-1}}$ is also $H_\pi$.

\begin{theorem} 
\label{thm-intertwiners-constant-in-some-directions}
Let $\pi$ be an irreducible unitary representation of $L_P$ with real infinitesimal character \textup{(}see the remark below\textup{)}. With the notation and assumptions given above, there is a strongly continuous family of unitary operators 
\[
\mathscr{A}_\nu\colon  L^2 (K,H_{\pi})^{K\cap L_P}  \longrightarrow L^2 (K,H_{\pi})^{K\cap L_P}
\qquad (\nu \in \mathfrak{a}_P^*),
\]
such that: 
\begin{enumerate}[\rm (i)]
\item For every $\nu\in \mathfrak{a}_P^*$, the operator $\mathscr{A}_\nu$ intertwines the compact models of the representations $\Ind_P^G e^{i\nu}\cdot\pi$ and  $\Ind_P^G w(e^{i\nu}\cdot\pi)$:
\begin{equation*}
\mathscr{A}_\nu \cdot  \bigl (\Ind_P^G e^{i\nu}\cdot\pi\bigr )(g)= \bigl (\Ind_P^G w(e^{i\nu}\cdot \pi)\bigr )(g) \cdot \mathscr{A}_\nu
\qquad \forall g\in G.
\end{equation*}

\item If $\mu,\nu \in   \mathfrak{a}_P^*$, and if $w(\mu) = \mu$, then $\mathscr{A}_\nu = \mathscr{A}_{\nu+ \mu}$.  Thus the family $\{\, \mathscr{A}_\nu:\nu\in \mathfrak{a}_P^*\,\}$ is constant in all but one direction within $\mathfrak{a}_P^*$.

\end{enumerate}
\end{theorem}

\begin{remark} 
See Definition~\ref{def-real-infinitesimal-character} for a rapid review of the concept of real infinitesimal character.  Knapp and Stein developed their theory intertwining operators, which we use in the proof, under the assumption that the inducing representation $\pi$ has real infinitesimal character,\footnote{See for instance \cite[Lemma 5.5]{KnappSteinII}, which makes it clear that the \emph{Basic Assumption} of Knapp and Stein is the assumption of real infinitesimal character.} 
which is the reason we assume it in the statement of the theorem.
\end{remark}

\begin{proof}[Proof of Theorem~\ref{thm-intertwiners-constant-in-some-directions}]

To begin, we shall construct an intermediate standard parabolic subgroup $Q$:
\[
P\subseteq Q\subseteq G.
\]
Let $\mathfrak{a}_Q\subseteq \mathfrak{a}_P$ be the orthogonal complement in $\mathfrak{a}_P$ of the element $X_\varphi$ in \eqref{eq-definition-of-x-phi-and-w}, and let $L_Q$ be the centralizer of $\mathfrak{a}_Q$ in $G$.  Observe that $L_P\subseteq L_Q$.  Define $N_Q$ to be the closed, connected subgroup of  $N_P$ with Lie algebra

\begin{equation}
    \label{eq-lie-algebra-of-n-q}
\mathfrak{n}_Q = \bigoplus_{\substack{\alpha \in \Delta^+(\mathfrak{g},\mathfrak{a}) \\ \alpha\vert_{\mathfrak{a}_Q \ne 0}}} \mathfrak{g}_\alpha, 
\end{equation}
and set $Q = L_Q N_Q=M_QA_QN_Q$.  This is the intermediate standard parabolic subgroup that we require.

The group 
\[
P{\cap} L_Q = L_P \cdot  (N_P{\cap} L_Q)
\]
is a parabolic subgroup of $L_Q$ with Levi factor $L_P$, and we can therefore form the parabolically induced representations 
\[
\Ind_{P{\cap}L_Q}^{L_Q} H_{e^{i\nu}\cdot \pi}\quad \text{and} \quad\Ind_{P{\cap}L_Q}^{L_Q} H_{w(e^{i\nu}\cdot \pi)}
\qquad (\nu \in \mathfrak{a}_P^*) ,
\]
which are unitary representations of $L_Q$.

The subgroup $\Phi[\mathrm{SL}(2,\R)]$  of $G$ is included in $L_Q$, thanks to the assumption \eqref{eq-conditions-on-x-phi}, and therefore  $w\in K\cap L_Q$.  Since $w$ normalizes $\mathfrak{a}_P$, it follows from the work of Knapp and Stein  \cite[Prop.\,8.6]{KnappSteinII} that there exists a strongly continuous family of unitary intertwining operators of $L_Q$-representations
\begin{equation}\label{eq-knapp-stein-intertwiners-for-l-prime-reps}
\mathscr{U}'_\nu\colon \Ind_{P{\cap}L_Q}^{L_Q}H_{e^{i\nu}\cdot \pi} \longrightarrow \Ind_{P{\cap}L_Q}^{L_Q}H_{w(e^{i\nu}\cdot \pi)},\end{equation} where $\nu \in \mathfrak{a}_P^*$.

Now by Lemma~\ref{lem-parabolically-induced-reps-on-the-split-component} the two  unitary representations of $L_Q$ that appear in \eqref{eq-knapp-stein-intertwiners-for-l-prime-reps}, when restricted to $M_Q$, depend only on $\pi$ and the value of the linear functional $\nu$ on $X_\varphi$, and when restricted to the subgroup $A_Q\subseteq L_Q$ these representations take the forms
\[
a \cdot f = e^{i\nu (a)}\cdot f \qquad \forall a\in A_Q,\,\,\forall f \in \Ind_{P{\cap}L_Q}^{L_Q}H_{e^{i\nu}\cdot \pi} 
\]
and 
\[
a \cdot f = e^{iw(\nu) (a)}\cdot f \qquad \forall a\in A_Q,\,\,\forall f \in \Ind_{P{\cap}L_Q}^{L_Q}H_{w(e^{i\nu}\cdot \pi)}.
\]
So the group $A_Q$ acts on both representation spaces through scalar multiplication. And in fact since $w(\nu)$ and $\nu$ are equal on $\mathfrak{a}_Q\subseteq \mathfrak{a}_P$, the action of $A_Q$ occurs through scalar multiplication by the \emph{same} character on the two representations.  It follows that for $\mu\in \mathfrak{a}_P^*$,  if $\mu(X_\varphi)=0$, or equivalently if $w(\mu) = \mu$, then the intertwining operator $\mathscr{U}'_\nu$ in \eqref{eq-knapp-stein-intertwiners-for-l-prime-reps} is also an intertwining operator 
\[
\mathscr{U}'_\nu\colon \Ind_{P{\cap}L_Q}^{L_Q}H_{e^{i(\nu+\mu)}\cdot \pi} \longrightarrow \Ind_{P{\cap}L_Q}^{L_Q}H_{w(e^{i(\nu+\mu)}\cdot \pi)}.
\]
 So we may now adjust the definition of the intertwiners $\mathscr{U}'_\nu$, if necessary, to ensure that 
 \begin{equation}
     \label{eq-constancy-condition-for-a-prime-intertwiners}
\mu\in  \mathfrak{a}_P \,\,\,\text{and}\,\,\,  \mu(X_\varphi) = 0 
\\
\quad \Rightarrow \quad \mathscr{U}'_\nu = \mathscr{U}'_{\nu+\mu}\quad \forall \nu\in  \mathfrak{a}_P^*.
 \end{equation}
If $L_Q$ was equal to $G$, then at this point the proof of the theorem would be complete. In any case, we shall assume \eqref{eq-constancy-condition-for-a-prime-intertwiners} from now on, and proceed to the general case.

We may use the functoriality of parabolic induction to construct from the $\mathscr{U}'_\nu$ a family of intertwining operators
\[
\mathscr{U}''_\nu \colon \Ind_{Q}^G\Ind_{P{\cap}L_Q}^{L_Q}H_{e^{i\nu}\cdot \pi}
\longrightarrow  \Ind_{Q}^G\Ind_{P{\cap}L_Q}^{L_Q}H_{w(e^{i\nu}\cdot \pi)}
\]
depending on $\nu \in \mathfrak{a}_P^*$ (thus $\mathscr{U}''_\nu$ is by definition obtained by applying the parabolic induction functor   to $\mathscr{U}'_\nu$). In the diagram 
\begin{equation}
    \label{eq-commuting-diagram-of-intertwiners}
{\footnotesize
\xymatrix@C=10pt{
\Ind_{Q}^G\Ind_{P{\cap}L_Q}^{L_Q}H_{e^{i\nu}\cdot \pi}\ar[d]_{\cong}\ar[r]^-{\mathscr{U}''_\nu}  & \Ind_{Q}^G\Ind_{P{\cap}L_Q}^{L_Q}H_{w(e^{i\nu}\cdot \pi)}\ar[d]^{\cong}
 \\
  L^2(  K,L^2(K {\cap} L_Q, H_\pi)^{K {\cap} L_P})^{K\cap L_Q} \ar[r]_-{\mathscr{A}''_\nu}& 
    L^2(  K,L^2(K {\cap} L_Q, H_{w(\pi)})^{K {\cap} L_P})^{K{\cap} L_Q}
}}
\end{equation}
that arises from functoriality of the induction in stages isomorphism in  Lemma~\ref{lem-induction-in-stages-in-compact-model}, the unitary isomorphism $\mathscr{A}''_\nu$ satisfies $\mathscr{A}_{\nu}'' = \mathscr{A} ''_{\nu+\mu}$ if $\mu(X_\varphi)=0$, because the operators $\mathscr{U}'_\nu$, and hence the operators $\mathscr{U}''_\nu$, have this property.
Consider now the diagram of unitary isomorphisms
\[
{\footnotesize
\xymatrix@C=-15pt@R=20pt{
& \Ind_{Q}^G\Ind_{P{\cap}L_Q}^{L_Q}H_{w(e^{i\nu}\cdot\pi)} \ar[rr]^{\cong} \ar'[d][dd]
& &  \Ind_P^G H_{w(e^{i\nu}\cdot \pi)}  \ar[dd]
\\ \Ind_{Q}^G\Ind_{P{\cap}L_Q}^{L_Q}H_{e^{i\nu}\cdot \pi} \ar[ur]^-{\mathscr{U}''_\nu}\ar[rr]^{\cong}\ar[dd]
& & \Ind_P^G H_{e^{i\nu}\cdot \pi} \ar[ur]^-{\mathscr{U}_\nu}\ar[dd]
\\
&L^2 (K,L^2(K \cap L_Q, H_\pi)^{K \cap L_P})^{K\cap L_Q}  \ar'[r]_-{\cong}[rr]
& &  L^2 (K,  H_\pi)^{K \cap L_P}
\\
L^2 (K,L^2(K \cap L_Q, H_\pi)^{K \cap L_P})^{K\cap L_Q} \ar[rr]_\cong  \ar[ur] _-{\mathscr{A}''_\nu}
& & L^2 (K,  H_\pi)^{K \cap L_P} \ar[ur]_{\mathscr{A}_\nu}
}
}
\]
 in which 
\begin{enumerate}[\rm (i)]
    \item the horizontal morphisms are the induction in stages isomorphisms from Lemmas~\ref{lem-induction-in-stages} and \ref{lem-induction-in-stages-in-compact-model}, and the front and rear faces are as described in Lemma~\ref{lem-induction-in-stages-in-compact-model},
    \item the downward maps are restriction isomorphisms  in \eqref{eq-compact-model-Hilbert-space-isomorphism}, 
    \item the left face  is \eqref{eq-commuting-diagram-of-intertwiners}, 
    \item $\mathscr{U}_\nu$ is chosen to make to top face commute, and $\mathscr{A}_\nu$ is chosen to make the bottom face commute.
\end{enumerate}
All faces other than the right-hand face commute  by construction, so the right-hand face commutes, too. Moreover, if $s(\mu)=\mu$, then $\mathscr{A}_{\nu} = \mathscr{A}_{\nu+\mu}$ by construction of $\mathscr{A}_\nu$.
\end{proof}

\subsection{\texorpdfstring{Cuspidal principal series and intertwining groups}{Cuspidal principal series and intertwining groups}}
\label{sec-intertwining-groups}

Let $P$ be a \emph{cuspidal} standard parabolic subgroup, that is,  assume that the  group $M_P$ in the  decomposition $P=M_P\cdot A_P\cdot N_P$ admits irreducible square-integrable   representations. Let $\sigma$ be one of these \emph{discrete series} representations and let $\nu\in\mathfrak{a}^*_P$. The pair $(\sigma,\nu)$ determines  a representation of $L_P=M_PA_P$: 
\[
\begin{array}{rccl}
\sigma\otimes e^{i\nu}:&L_P&\longrightarrow&U(H_\sigma)
\\
&ma&\longmapsto&e^{i\nu}(a)\sigma(m) .
\end{array}
\] 

\begin{remark} 
Earlier, in \eqref{eq-notation-e-to-the-i-nu-times-tau-1}, we used the notation $e^{i\nu} \cdot \pi$ in a similar context.  But there $\pi$ was a representation of $L_P$, while here $\sigma$ is a representation of $M_P$, and indeed a square-integrable representation of $M_P$.  Throughout the paper we shall reserve the notation $\sigma\otimes e^{i \nu}$ for this particular case.
\end{remark}

\begin{definition}
\label{def-cuspidal_PS}
The \emph{cuspidal principal series representation} with parameters $(\sigma,\nu)$ is the parabolically induced representation 
\[
\pi_{P,\sigma,\nu}:=\Ind_P^G \sigma\otimes e^{i\nu}.
\]
\end{definition}

From now on we shall work with the compact models of the representations $\pi_{P,\sigma,\nu}$, and in an effort to streamline our notation, we shall write 
\begin{equation}
\label{eq-def-Ind-H_sigma-compact}
\Ind H_\sigma = L^2 (K, H_\sigma)^{K \cap L_P}
\end{equation} 
from here onwards. The Hilbert space $\Ind H_\sigma$ depends only on the restriction of the representation $\sigma$ to the compact group $K\cap L_P$, and not on $P$, although the representation $\pi_{P,\sigma,\nu}$ on it \emph{does} depend on
$P$, as well as  on $\sigma$ and $\nu$, of course.

The Knapp-Stein theory of intertwining operators \cite{KnappSteinI,KnappSteinII} associates to each $w\in N_K(L_P)$, each discrete series representation $\sigma$, and each $\nu \in \mathfrak{a}_P^*$, a unitary operator
\begin{equation}
\label{def-A_w_sigma_nu}\mathscr{A}_{w,\sigma,\nu}: \Ind  H_{\sigma }\longrightarrow\Ind H_{w(\sigma) }
\end{equation}
that is an equivalence between the principal series representations $\pi_{P,\sigma,\nu}$ and $\pi_{P,w(\sigma), w(\nu)}$, in that  
\begin{equation*}
\mathscr{A}_{w,\sigma,\nu}\cdot\pi_{P,\sigma,\nu}(g)  = \pi_{P,w(\sigma), w(\nu)}(g)\cdot \mathscr{A}_{w,\sigma,\nu}\qquad \forall g \in G.
\end{equation*}
The operator $\mathscr{A}_{w,\sigma,\nu}$ varies strongly-continuously with $\nu\in \mathfrak{a}^*$. Moreover if $\sigma$ is any discrete series representation, then 
\begin{multline}
\label{eq-cocycle-relation}
\mathscr{A}_{w_1,w_2(\sigma),w_2(\nu)}\mathscr{A}_{w_2,\sigma,\nu} =  \mathscr{A}_{w_1w_2,\sigma, \nu}
\\
\forall \nu \in \mathfrak{a}_P^*,\,\, \forall w_1,w_2\in N_K(L_P).
\end{multline}

Now if $w\in N_K(L_P)$ not only normalizes $L_P$, but actually centralizes $L_P$, then $\mathscr{A}_{w,\sigma,\nu} = \mathrm{id}$.  
This prompts the following definition.

\begin{definition}\label{def-int-group_W_sigma} 
The \emph{intertwining  group} associated with the pair $(P,\sigma)$ is the finite group 
\[
\begin{aligned}
W_\sigma
     & = \bigl \{w \in N_K(L_P) : \operatorname{Ad}_w ^* \sigma \simeq \sigma\, \bigr \} \big /  Z_K(L_P)
     \\
     & = \bigl \{w \in N_K(\mathfrak{a}_P) : \operatorname{Ad}_w ^* \sigma \simeq \sigma\, \bigr \} \big /  Z_K(\mathfrak{a}_P)
\end{aligned} 
\]
\end{definition}

The intertwining group acts faithfully as orthogonal transformations of $\mathfrak{a}_P$, and we shall typically view it as a group of automorphisms of $\mathfrak{a}_P$.  But since the Knapp-Stein intertwiners associated to elements of  $Z_K(L_P)$ are trivial, it follows from 
\eqref{eq-cocycle-relation} that we may associate to any $w\in W_\sigma$ a well-defined intertwining operator 
\[
\mathscr{A}_{w,\sigma,\nu} \colon \Ind H_\sigma \longrightarrow \Ind H_{w(\sigma)}
\]
by choosing any representative of $w$ in $N_K(L)$.

\subsection{R-groups}
\label{sec-r-groups}

For  $w\in W_\sigma$, choose a unitary equivalence 
\begin{equation*}
H_{w(\sigma)} \stackrel \simeq \longrightarrow H_{ \sigma}
\end{equation*}
of representations of $M_P$, and note that the same operator is also a unitary equivalence
\[
  H_{\sigma} \stackrel \simeq \longrightarrow H_{w^{-1}( \sigma)}.
\]
We shall denote  the induced unitary equivalences by 
\begin{equation}
    \label{eq-equivalence-of-reps-e}
E \colon \Ind H_{w(\sigma)} \stackrel \simeq \longrightarrow \Ind H_{\sigma}
\quad \text{and} \quad 
E \colon \Ind H_{\sigma} \stackrel \simeq \longrightarrow \Ind H_{w^{-1}(\sigma)} 
\end{equation}
(they are one and the same linear map). Note that by Schur's lemma, these operators are unique up to multiplication by a complex scalar of modulus $1$. 

It is proved in \cite[Prop.\,8.6(ii)]{KnappSteinII}  that 
\begin{equation}
    \label{eq-knapp-stein-a-and-e-relation}
 E   \circ\mathscr{A}_{w,\sigma,\nu}  = \mathscr{A}_{w,w^{-1}(\sigma),\nu}\circ  E 
\colon \Ind H_\sigma \longrightarrow \Ind H_\sigma ,
\end{equation}
and it follows from this that the adjusted Knapp-Stein operators

\begin{equation}
    \label{eq-adjusted-knapp-stein-ops}
 \mathscr{A}'_{w,\sigma,\nu} = E   \circ\mathscr{A}_{w,\sigma,\nu}\colon 
  \Ind H_{\sigma }\longrightarrow\Ind   H_{\sigma},
\end{equation}
which are well-defined up to multiplication by a complex scalar of modulus $1$, satisfy the projective cocycle relation 
\begin{multline}
\label{eq-porjective-cocycle-relation}
\mathscr{A}'_{w_1,w_2(\sigma),w_2(\nu)}\mathscr{A}'_{w_2,\sigma,\nu} =  \text{scalar}\cdot \mathscr{A}'_{w_1w_2,\sigma, \nu}
\\
\forall \nu \in \mathfrak{a}_P^*,\,\, \forall w_1,w_2\in W_\sigma,
\end{multline} 
for a scalar of modulus $1$.

\begin{definition}\label{def-W_0sigma}
We shall denote by  $W_{0,\sigma} \subseteq W_\sigma$ 
the normal subgroup of $W_\sigma$ consisting of those elements $w$ for which  the adjusted Knapp-Stein operator 
\[
\mathscr{A}'_{\sigma, w,0} \colon \Ind H_{\sigma }\longrightarrow\Ind   H_{\sigma}
\]
(that is, the adjusted Knapp-Stein operator for $\nu {=} 0$) is a scalar multiple of the identity operator. 
\end{definition}

Knapp and Stein gave  a complete description of the intertwining group $W_\sigma$ in terms  $W_{0,\sigma}$   and their so-called $R$-group  \cite{KnappSteinI,KnappSteinII}.
But rather than present the Knapp-Stein $R$-group, we shall follow a hybrid approach   that combines the Knapp-Stein theory with Vogan's theory of the $R$-group  \cite[Ch.\,4]{VoganGreenBook}.  The following theorem summarizes what we shall need for this paper.

\begin{theorem} 
\label{thm-Rgroup-sl2s}
Let $P=M_PA_PN_P$ be a standard parabolic subgroup of $G$ and let $\sigma$ be a discrete series representation of $M_P$.
There exists a reduced root system $\overline{\Delta}_0$ on  $\mathfrak{a}^*_P$ with Weyl group $W(\overline{\Delta}_0)$ such that 
\begin{enumerate}[\rm (i)]
\item 
$W_{0,\sigma}=W(\overline{\Delta}_0)$ 
as groups of automorphisms of $\mathfrak{a}^*_P$. 
\end{enumerate}
Moreover there is  a system of positive roots $\overline \Delta_0^+\subseteq  \overline\Delta_0$ and 
 a morphism of Lie algebras
\[
\varphi \colon \mathfrak{sl}(2,\R)\times\cdots\times\mathfrak{sl}(2,\R)\longrightarrow\mathfrak{g}
\] and 
such that, if $\varphi_j$ is the restriction of $\varphi$ to the $j$\textsuperscript{th} factor, and if 
\[
X_j = \varphi_j \left ( \left [
\begin{smallmatrix}
    1 & 0 \\ 0 & -1
\end{smallmatrix}
\right ]\right ) 
\quad \text{and} \quad
s_j = \Phi_j \left ( \left [
\begin{smallmatrix}
    0 & -1 \\ 1 & 0
\end{smallmatrix}
\right ]\right )
 ,
\]
\textup{(}where $\Phi_j$ is the Lie group morphism induced from $\varphi_j$\textup{)}, then
\begin{enumerate}[\rm (i)]
\setcounter{enumi}{1}
    \item  $X_j$ is an element of $\mathfrak{a}_P$, for every $j$, and the image $\varphi_j [\mathfrak{sl}(2,\R)]\subseteq \mathfrak{g}$ commutes with the orthogonal complement of $X_j$ in $\mathfrak{a}_P$,  

    \item $s_j\in N_K(\mathfrak{a}_P)$ for all $j$, 
    
    \item if $S_\sigma$ is the  abelian group of   orthogonal transformations of $\mathfrak{a}^*_P$ generated by   the $\Ad^*_{s_j} $, then   the group
   \[
   R_\sigma=\bigl \{\, w\in W_\sigma \,\big \vert \,  w(\overline{\Delta}_0^+)\subseteq \overline{\Delta}_0^+\bigr \} 
   \]
   is a subgroup of  $S_\sigma$, and 

   \item the intertwining group $W_\sigma $ decomposes as a semidirect product 
    \[
    W_\sigma = W_{0,\sigma}\rtimes R_\sigma.
    \]
\end{enumerate}
\end{theorem}

\begin{proof}
Vogan constructs the representations $\pi_{P,\sigma, \nu}$ not by parabolic induction, but by cohomological induction, starting from what he terms a set of \emph{discrete $\theta$-stable data} $(\mathfrak{q},H,\delta)$; see \cite[Defn.\,6.5.1]{VoganGreenBook}, as well as \cite[Sec.\,3]{CHS} for a survey that is suited to the perspective of this paper.  He defines an intertwining group $W_\delta$ \cite[Thm\,4.4.8]{VoganGreenBook} that turns out to be equal to the Knapp-Stein group $W_\sigma$ as a group of automorphisms of $\mathfrak{a}_P$ \cite[Lemma\,5.2]{CHS}, and he defines a subgroup $W_\delta^0\subseteq W_\delta$ in \cite[Defn\,4.3.13]{VoganGreenBook}. Delorme proves that $W_\delta^0$ is equal to the Knapp-Stein group that we are denoting $W_{0,\sigma}$ \cite[Thm\,1]{Delorme84}.

By definition, Vogan's $W^0_\delta$ is the Weyl group of a root system $\overline\Delta_\delta$, described in \cite[Defn\,4.3.11]{VoganGreenBook}, and we take this to be our $\overline \Delta_0$. This settles item (i) in the statement of the theorem.

Vogan also defines a set $\overline{\Delta_S}\subseteq  \mathfrak{a}_P^*$ in \cite[(4.3.26)]{VoganGreenBook}. He proves in \cite[Lemma\,4.3.27]{VoganGreenBook} that the simple reflections in these elements  pairwise commute. We take these simple reflections to be our $s_j$.  Vogan calls the group that they generate  $W(\overline{\Delta}_S)$; this is our  group  $S_\sigma$.

The fact that the reflections $s_j$ arise from a Lie algebra morphism $\varphi$, as in items (ii) and (iii) in the statement of the theorem, is a consequence of the strong orthogonality property of $\overline \Delta_S$ proved in \cite[Lemma\,4.3.7]{VoganGreenBook} and the construction in \cite[4.3.6]{VoganGreenBook}.

Finally, we choose any system of positive roots in $\overline \Delta_0$ to be our $\overline \Delta_0^+$.    Items (iv) and (v) in the statement of the theorem are proved   \cite[Lemma\,4.3.29]{VoganGreenBook}.
\end{proof}

\begin{remark}
It follows from (ii) that the Lie group morphisms $\Phi_j$  that are integrated from the Lie algebra morphisms $\varphi_j$ satisfy Condition \eqref{eq-conditions-on-x-phi}. This will allow us to apply Theorem \ref{thm-intertwiners-constant-in-some-directions} at a  critical point in Section~\ref{sec-rescaling-automorphisms}  (see Lemma \ref{lem-u-mu-constant-in-some-directions}).
\end{remark}

\section{The Mackey bijection}
\label{sec-mackey-bijection}

Afgoustidis \cite{AfgoustidisMackeyBijection} used Vogan's theory of minimal $K$-types, as presented in   \cite{VoganGreenBook}, to 
construct a 
\emph{Mackey bijection} between the tempered dual of a real reductive group and the unitary dual of the associated Cartan motion group.  We shall review Afgoustidis's construction in this section.

\subsection{Imaginary part of the infinitesimal character}
Let $G$ be a real reductive group, equipped with a  Cartan decomposition $  \mathfrak{g}  =  \mathfrak{k}\oplus \mathfrak{s}$.
If $\mathfrak{h}_{\C}$ is any Cartan subalgebra of the complexification of the Lie algebra $G$, then it is well known that every irreducible unitary representation $\pi$ of $G$ has an \emph{infinitesimal character} 
\[
\InfCh (\pi) \in \mathfrak{h}^* / W(\mathfrak{g}_{\C},\mathfrak{h}_{\C})
\]
Vogan explains in \cite{Vogan00} how to define a continuous map 
\[
\mathfrak{h}^* / W(\mathfrak{g}_{\C},\mathfrak{h}_{\C}) \longrightarrow 
\mathfrak{a}^*/  W(\mathfrak{g},\mathfrak{a})
\]
(involving the maximal abelian subspace $\mathfrak{a}$ from the Iwasawa decomposition \eqref{eq-fixed-iwasawa-decomp}) that extracts from $\InfCh(\pi)$  what may be called the \emph{imaginary part of the infinitesimal character},
\[
\ImInfCh(\pi) \in \mathfrak{a}^*/  W(\mathfrak{g},\mathfrak{a}).
\]
See for instance \cite{BraddHigsonYuncken24} for an exposition of this concept. 
 
Theorem~\ref{thm-afgoustidis-characterization-of-tempered-reps-with-fixed-im-inf-ch} below  is the first fundamental result  required for the Mackey bijection. It is a version for tempered representations of a more general result, due to Knapp \cite{KnappRepTheorySemisimpleGroups} and Vogan \cite{Vogan00}, about general irreducible unitary representations, beyond tempered.   Again, see \cite{BraddHigsonYuncken24} for an exposition. 
 
\begin{definition}
\label{def-real-infinitesimal-character}
An irreducible unitary representation  $\tau$ of a real reductive group is said to have \emph{real infinitesimal character} if
\[
\ImInfCh (\tau) =0\in \mathfrak{a}^*/W(\mathfrak{g}, \mathfrak{a} ) .
\]
\end{definition}

\begin{definition}
A unitary representation of a real reductive group is \emph{tempiric}   if it is tempered, irreducible, and has real infinitesimal character in the sense of Definition~\ref{def-real-infinitesimal-character}.
\end{definition}

\begin{definition}
\label{def-equivalence-class-of-nu}
    Let $\nu \in \mathfrak{a}^*$. (i) We shall denote by $[\nu]$ the image of $\nu$ in the quotient  $\mathfrak{a}^*/W(\mathfrak{g}, \mathfrak{a} )$. (ii) By the \emph{centralizer of $\nu$ in $G$} we shall mean the centralizer in $G$ of the vector in $\mathfrak{a}$ that corresponds to $\nu$ under the isomorphism $\mathfrak{a}\cong \mathfrak{a}^*$   given by the inner product on  $\mathfrak{a}$ from  \eqref{eq-Cartan-form}. (iii) If $\mathfrak{a}_P\subseteq \mathfrak{a}$ and $\nu\in \mathfrak{a}_P^*$, then we  extend $\nu$ to a linear functional $\tilde \nu \in \mathfrak{a}^*$ that is zero on the orthogonal complement of $\mathfrak{a}_P$ in $\mathfrak{a}$, and set $[\nu]=[\tilde \nu] $. We also define the centralizer of $\nu$ in $G$ to be the centralizer of $\tilde \nu$. 
\end{definition}

\begin{theorem}[c.f.~{\cite[Proof of Thm.~3.5(c)]{AfgoustidisMackeyBijection}}]
    \label{thm-afgoustidis-characterization-of-tempered-reps-with-fixed-im-inf-ch}
Let  $\nu\in  \mathfrak{a}^*$, and let $L $ be the centralizer of $\nu$ in $G$. Let  $P{=}L_PN_P$ be the standard parabolic subgroup of $G$ whose Levi factor is $L_P=L$. The correspondence 
\[
\tau \longmapsto \Ind_P^G e^{i\nu}\cdot \tau  
\]
\textup{(}see \eqref{eq-notation-e-to-the-i-nu-times-tau-1} for the notation $ e^{i\nu}\cdot \tau$\textup{)} determines a bijection from the set of unitary equivalence classes of tempered irreducible unitary representations of $L_P$ with real infinitesimal character to the set of unitary  equivalence classes of tempered irreducible  unitary representations $\pi$ of $G$ for which $\ImInfCh(\pi) = [\nu]$. 
\end{theorem}

\begin{remark} 
In the theorem, the restriction of $\nu$ to the orthogonal complement of $\mathfrak{a}_P\subseteq \mathfrak{a}$ is zero, so we may equally well view $\nu$ as an element of $\mathfrak{a}_P^*$. 
\end{remark}

\subsection{Minimal K-types of tempiric representations}
The second  fundamental result that is required for the Mackey bijection---and that is the most difficult---is Theorem~\ref{thm-vogan-classification-of-tempiric-reps} below, due to Vogan. All of the ingredients for the proof can be found in Vogan's monograph \cite{VoganGreenBook}. For the actual statement, see for instance  \cite[Thm.~1.2]{Vogan07}.

\begin{definition} 
\label{def-norm-of-a-rep-of-k}
Let $K$ be a compact group (not necessarily connected) and   let $T$ be a  maximal torus in the identity component of $K$.  Fix a system of positive roots for $(\mathfrak{k}, \mathfrak{t})$, and let $\rho_K$ be the associated half-sum of positive roots.  Fix a $K$-invariant inner product on $\mathfrak{k}$.  If $\theta$ is any irreducible representation of $K$, then we define 
\[
\| \theta\| = \| \mu {+} 2 \rho_K\|,
\]
where $\mu$ is any highest weight of $\theta$. 
\end{definition}

\begin{remark}
Because $K$ need not be connected in Definition~\ref{def-norm-of-a-rep-of-k}, the highest weight chosen there   need not be unique. But the norm $\| \theta\|$ is independent of the choice. Nor does the norm depend on any of the other choices made in Definition~\ref{def-norm-of-a-rep-of-k}, except for the choice of inner product.  The notion of minimality in the theorem below is independent even of the choice of norm.
\end{remark}

\begin{theorem}[Vogan]
\label{thm-vogan-classification-of-tempiric-reps}
Let $G$ be a real reductive group with maximal compact subgroup $K$.  Every tempiric representation of $G$ has a unique $K$-type $\theta$ for which $\| \theta \|$ is minimal among all $K$-types in the representation.  This $K$-type has multiplicity $1$.  Every irreducible representation of $K$ is the \emph{minimal $K$-type}, in this sense, of a unique tempiric representation of $G$, up to unitary equivalence.
\end{theorem}

\subsection{The Cartan motion group}
In view of the Cartan decomposition  of $\mathfrak{g} = \mathfrak{k} \oplus \mathfrak{s}$, the group $G_0=K \ltimes (\mathfrak{g}/\mathfrak{k})$ in \eqref{eq-fiber-of-the-smooth-family} may be written as 
\[
G_0 \cong K \ltimes \mathfrak{s}.
\]
This is the \emph{Cartan motion group} associated to $G$ (and to the given Cartan decomposition).

Mackey proved that the irreducible unitary representations of the Cartan motion group $G_0$ (or indeed of any semidirect product of a compact group acting on a vector group) all arise from the following construction.

\begin{definition} 
    \label{def-rho-theta-nu}
    Let  $\nu\in \mathfrak{s}^*$, and let $\theta$ be an irreducible unitary representation of the isotropy group $K_\nu$ in $K$. We shall denote by $\rho_{\theta, \nu}$ the  following  unitary representation of the group $G_0$:
\begin{equation*}
\rho_{\theta, \nu}  = \Ind_{K_\nu\ltimes \lie{s}} ^{K\ltimes \lie{s}} \theta \otimes e^{ i  \nu} \colon K\ltimes \mathfrak{s} \longrightarrow U \bigl (L^2 (K, H_\theta)^{K_{\nu}}\bigr ).
\end{equation*}
Here, $\theta\otimes e^{i \nu} \colon (k,X)\longmapsto e^{i \nu (X)}\theta(k)$.\end{definition} 

\begin{remark}
\label{rem-extended-rho-notation}
Later we shall use the same notation $\rho_{\theta,\nu}$ in cases where  $\theta$ is \emph{any} unitary representation of $K_\nu$, irreducible or not.
\end{remark}

\begin{theorem}[{\cite[Sec.\,7]{Mackey49}}]
The representation $\rho_{\theta,\nu}$ above is   {irreducible}. Every irreducible unitary representation of $G_0$ is equivalent to some $\rho_{\theta,\nu}$, and  $\rho_{\theta,\nu}$ is unitarily equivalent to $\rho_{\theta',\nu'}$ if and only if the  pairs $(\theta,\nu)$ and $(\theta',\nu')$  are $K$-conjugate.
\end{theorem}

Now, the given inner product on $\mathfrak{s}$ determines an inner product on $\mathfrak{s}^*$ and then a bijection
\[
\mathfrak{a}^*/W(\mathfrak{g}, \mathfrak{a})
\stackrel \cong \longrightarrow \mathfrak{s}^* / K 
\]
in which a linear functional $\mathfrak{a}$  is extended to $\mathfrak{s}$ by requiring it to be zero on the orthogonal complement of $\mathfrak{a}\subseteq \mathfrak{s}$. Using the bijection  we obtain:

\begin{theorem}
\label{thm-final-mackey-description-of-motion-group-dual}
Let $G$ be a real reductive group with Lie-algebraic Cartan decomposition $\mathfrak{g} = \mathfrak{k}\oplus \mathfrak{s}$, and let $\mathfrak{a}$ be a maximal abelian subspace of $\mathfrak{s}$.  There is a unique bijection
\[
\widehat {K\ltimes \mathfrak{s}} \cong  \Bigl ( \bigsqcup _{\nu \in \mathfrak{a}^*} \widehat K_\nu \Bigr ) \Big / W(\mathfrak{g},\mathfrak{a})
\]
that corresponds the equivalence class of the representation $\rho_{\theta,\nu}$ in Definition~\ref{def-rho-theta-nu}    to the class of $(\theta, \nu)$.
\qed
\end{theorem}

\subsection{Afgoustidis's Mackey bijection}

The Mackey bijection of Afgoustidis \cite{AfgoustidisMackeyBijection} is  easily constructed from the ingredients that we have reviewed above.  

\begin{theorem}[{\cite[Thm.~3.5]{AfgoustidisMackeyBijection}}]
There is a unique bijection 
\begin{equation*}
\widehat G_0  \stackrel \cong \longrightarrow  \Gtemp
\end{equation*}
that maps the equivalence class of the representation $\rho_{\theta,\nu}$ in Definition~\textup{\ref{def-rho-theta-nu}} to the equivalence class of the irreducible tempered unitary representation 
\[
\Ind_P^G e^{i\nu}\cdot \tau ,
\]
where $P$ is any parabolic subgroup whose Levi factor $L$  is the centralizer of $\nu$, and $\tau$ is a tempiric representation of $L$ whose minimal $K$-type is $\theta$.  
\end{theorem}

\section{Rescaling automorphisms of the group C*-algebra}
\label{sec-rescaling-automorphisms}

\subsection{Rescaling automorphisms of  the tempered dual} 
\label{sec-rescaling-maps-on-tempered-dual}
The following remarks are meant to serve as an informal motivation for the construction that we shall carry out in the rest of Section~\ref{sec-rescaling-automorphisms}.

Theorem~\ref{thm-afgoustidis-characterization-of-tempered-reps-with-fixed-im-inf-ch} shows that every tempered irreducible unitary representation of $G$ is equivalent to   an essentially unique  $\Ind_P^G e^{i\nu}\cdot \tau$,  where  $P=L_PN_P=M_PA_PN_P$ is a standard parabolic subgroup whose Levi factor is the centralizer of $\nu$, and where $\tau$ is a tempiric representation of $L_P$.  With this, we may define \emph{rescaling automorphisms} 
\begin{equation*}
    \label{eq-rescaling-maps-on-tempered-dual}
\begin{gathered}
\alpha_t \colon \Gtemp
\longrightarrow \Gtemp 
\\
 \alpha_t \colon \bigl [\Ind_P^G e^{i\nu}\cdot\tau\bigr ] \longmapsto  \bigl [\Ind_P^G e^{it^{-1}\nu}\cdot\tau\bigr ]
\end{gathered}
\qquad \forall t>0 .
\end{equation*}
These are bijections from the tempered dual to itself. In fact they are homeomorphisms.\footnote{We shall not make use of the fact that the rescaling maps are continuous, but the tempered dual, as a topological space, is the spectrum of the reduced group $C^*$-algebra, and continuity of the rescaling map can be deduced from what follows.} 

In effect, our  goal in this section will be to lift the rescaling automorphisms   above  from the tempered dual, which may be identified with the spectrum of the reduced group $C^*$-algebra $C^*_r(G)$ (see below), to automorphisms of the $C^*$-algebra itself.  But our initial focus will be on cuspidal parabolic induction, rather than parabolic induction from more general tempiric representations: we shall construct a   one-parameter group of $C^*$-algebra automorphisms 
\[
\alpha_t \colon C^*_r (G) \longrightarrow C^*_r (G) \qquad (t>0)
\]
such that 
\[
\pi_{P,\sigma,\nu}\circ \alpha_t \simeq \pi_{P,\sigma,t^{-1}\nu}
\]
for every $t>0$, and every $\pi_{P,\sigma,\nu}$.  We shall eventually return to parabolic induction from tempiric representations, but only in Section~\ref{sec-bijection-characterization}; see Lemma~\ref{lem-composition-of-pi-with-alpha-1}.

\subsection{Structure of the reduced group C*-algebra} 
\label{sec-structure-of-the-c-star-algebra}
Let $G$ be  a real reductive group.  In this section we shall quickly review some results from \cite{CCH16} and \cite{ClareHigsonSongTang24} about the structure of the reduced $C^*$-algebra of $G$.

Let $\pi$ be a tempered and admissible unitary representation of $G$. The formula
\begin{equation}
    \label{eq-integration-of-a-rep}
\pi (\xi)\phi=\int_G \xi(g) \pi (g)\phi \qquad \forall \xi\in C_c^\infty (G),\,\,\forall \phi \in H_\pi.
\end{equation}
defines an associated representation of the associative convolution algebra $C_c^\infty (G)$ as bounded, and in fact compact, operators on the Hilbert space $H_\pi$.  Now, as usual, denote by $C^*_r(G)$ the norm-comp\-l\-etion of $C_c^\infty (G)$ as a convolution algebra of operators on $L^2 (G)$; this is the reduced group $C^*$-algebra of $G$ \cite[Sec.\,7.2]{Pedersen79}. The  representation \eqref{eq-integration-of-a-rep} extends to a representation of $C^*_r(G)$ as compact operators on $H_\pi$,
\begin{equation}
    \label{eq-integration-to-a-c-star-rep}
\pi  \colon C^*_r (G) \longrightarrow \mathfrak{K} (H_\pi ),
\end{equation}
and in this way a bijection is determined from unitary equivalence classes of tempered irreducible unitary representations of $G$ to irreducible representations of $C^*_r (G)$ \cite{CowlingHaagerupHowe}. In other words Harish-Chandra's tempered dual is precisely the same as the $C^*$-algebraic reduced dual.

The occurrence of tempered representations in families leads to more elaborate representations than \eqref{eq-integration-to-a-c-star-rep}:

\begin{theorem}[See for example {\cite[Cor.\;4.12]{CCH16}}]
\label{thm-reductive-riemann-lebesgue}
Let $G$ be  a real reductive group, let $P=M_PA_PN_P$ be a parabolic subgroup, and let $\sigma $ be a discrete series representation of $M_P$.   There is a  $C^*$-algebra homomorphism
\[
  \pi_{P,\sigma} \colon C_r^*(G) \longrightarrow C_0(\mathfrak{a}_P^*,\compop(\Ind H_\sigma )  )
\]
\textup{(}the target is the $C^*$-algebra of norm-continuous functions on $\mathfrak{a}_P^*$, vanishing at infinity,  with values in the compact operators on $\Ind H_\sigma$\textup{)} such that $\pi_{P,\sigma}(\xi)(\nu) = \pi_{P,\sigma,\nu}(\xi)$  for every $\nu\in \lie{a}^*_P$.
\end{theorem}

The image of the $C^*$-algebra homomorphism in the theorem may be described using the Knapp-Stein theory of intertwining operators.  For $w\in N_K(L_P)$, define a homomorphism of $C^*$-algebras 
\[
a_{w,\sigma}\colon C_0\bigl (\mathfrak{a}_P^*, \Compacts(\Ind H_\sigma)\bigr )
\longrightarrow 
C_0\bigl (\mathfrak{a}_P^*, \Compacts(\Ind H_{w(\sigma)})\bigr ),
\]
by
\begin{equation}
    \label{eq-c-star-action-of-intertwiners}
a_{w,\sigma}(f)(w(\nu)) = \mathscr{A}_{w,\sigma,\nu} f(\nu) \mathscr{A}_{w,\sigma,\nu}^*
\qquad   \forall \nu \in \mathfrak{a}^*_P
\end{equation}
The cocycle relation \eqref{eq-cocycle-relation} for the Knapp-Stein operators implies that 
\begin{equation}
\label{eq-c-star-action-of-intertwiners-2}
a_{w_1,w_2(\sigma)}\circ a_{w_2,\sigma}= a_{w_1w_2,\sigma}
\end{equation}
In addition,  if $w\in W_\sigma$, then the formula \begin{equation}
    \label{eq-w-sigma-action-on-c-star-algebra-1}
w(f)(w(\nu)) = \mathscr{A}'_{w,\sigma,\nu} f(\nu) \mathscr{A}'{}^*_{\!\!\!w,\sigma,\nu}
\qquad   \forall \nu \in \mathfrak{a}^*_P,
\end{equation}
involving the adjusted Knapp-Stein operators in \eqref{eq-adjusted-knapp-stein-ops}, gives a well-defined\footnote{Recall that the adjusted Knapp-Stein operator $\mathscr{A}'_{w,\sigma,\nu}$ itself is well-defined only up to multiplication by a complex scalar of modulus $1$.} automorphism 
\begin{equation}
    \label{eq-w-sigma-action-on-c-star-algebra-2}
w\colon  C_0\bigl (\mathfrak{a}_P^*, \Compacts(\Ind H_\sigma)\bigr ) \longrightarrow  C_0\bigl (\mathfrak{a}_P^*, \Compacts(\Ind H_\sigma)\bigr ).
\end{equation}
We obtain in this way an action of $W_\sigma$ by automorphisms on the $C^*$-algebra $C_0 (\mathfrak{a}_P^*, \Compacts(\Ind H_\sigma) )$.

\begin{theorem}[See {\cite[Thm.\,6.8]{CCH16}}] 
\label{thm-structure-of-reduced-c=star-algebra-0}
The image of the $C^*$-algebra homomorphism in Theorem~\textup{\ref{thm-reductive-riemann-lebesgue}} is 
\[
  \pi_{P,\sigma}\bigl [C_r^*(G)\bigr ]   =   C_0  (\mathfrak{a}^*_P ,\compop(\Ind H_\sigma )  )^{W_\sigma} .
\]
\end{theorem}

\begin{definition}\label{def-associate}
Two pairs $(P_1,\sigma_1)$ and $(P_2,\sigma_2)$, each consisting of a parabolic subgroup of $G$ and a discrete series representation of the $M$-part of the parabolic subgroup, are said to be \emph{associate} if there exists an element of $K$ that conjugates the Levi factor of $P_1$ to the Levi factor of $P_2$, and conjugates the equivalence class of $\sigma_1$ to the equivalence class of $\sigma_2$.  
An equivalence class   under the relation in Definition~\ref{def-associate} is called an \emph{associate class}. 
\end{definition}

The morphisms in Theorem~\ref{thm-structure-of-reduced-c=star-algebra-0} may be combined to give a complete description of the reduced $C^*$-algebra of $G$, as follows.

\begin{theorem}[See {\cite[Thm.\,6.8]{CCH16}} again] 
\label{thm-structure-of-reduced-c=star-algebra-1}
The $C^*$-algebra homomorphisms in Theorem~\textup{\ref{thm-structure-of-reduced-c=star-algebra-0}} assemble into a $C^*$-algebra isomorphism
\[
 \bigoplus _{[P,\sigma]} \pi_{P,\sigma} \colon C_r^*(G)   \stackrel \cong \longrightarrow \bigoplus _{[P,\sigma]}   C_0  (\mathfrak{a}^*_{P},\compop(\Ind H_\sigma )  )^{W_\sigma}.
\]
The $c_0$-direct sum is taken over representatives of all the associate classes of pairs $(P,\sigma)$ for the group $G$.
\end{theorem}

As we noted in Section~\ref{sec-r-groups}, the group $W_\sigma$ is a semidirect product $W_{0,\sigma}$ and $R_\sigma$ where $W_{0,\sigma}$ is a Weyl group. We can divide $\lie{a}^*_P$ into closed Weyl chambers for the action of $W_{0,\sigma}$, and choose one  (closed)  \emph{positive  chamber} 
\begin{equation}
\label{eq-positive-sigma-weyl-chamber}
\lie{a}^*_{P,+}\subseteq \lie{a}^*_{P} ,
\end{equation}
which is a fundamental domain for the action of $W_{0,\sigma}$ on $\lie{a}_P^*$.

\begin{remark}
The division of $\mathfrak{a}^*_P$ into  chambers depends on $\sigma$ \textit{via} $W_{0,\sigma}$, although this is not reflected in the notation \eqref{eq-positive-sigma-weyl-chamber}. 
\end{remark}

The  morphisms $\pi_{P,\sigma}$ in Theorem~\ref{thm-reductive-riemann-lebesgue} determine morphisms
\begin{equation}
    \label{eq-restricted-pi-sigma-morphisms}
  \pi_{P,\sigma} \colon C_r^*(G)  \longrightarrow   C_0  (\mathfrak{a}^*_{P,+},\compop(\Ind H_\sigma )  )^{R_\sigma}
\end{equation}
 by restriction to the positive  chamber. These also assemble into a $C^*$-algebra isomorphism, which is a variation on the isomorphism in Theorem~\ref{thm-structure-of-reduced-c=star-algebra-1}, as follows:

\begin{theorem}[See for instance {\cite[Sec.\,2]{ClareHigsonSongTang24}}]
\label{thm-structure-of-reduced-c=star-algebra-2}
Let $G$ be  a real reductive group.  The morphisms \eqref{eq-restricted-pi-sigma-morphisms}  determine a $C^*$-algebra isomorphism
\[
\bigoplus_{[P,\sigma]} \pi_{P,\sigma} \colon C_r^*(G)\stackrel \cong \longrightarrow  \bigoplus_{[P,\sigma]} C_0\bigl (\mathfrak{a}^*_{P,+},\compop(\Ind H_\sigma )\bigr )^{R_\sigma} .
\]
\end{theorem}

\begin{proof}
It suffices to show that  the $C^*$-algebra morphism 
\[
 C_0  (\mathfrak{a}^*_{P},\compop(\Ind H_\sigma )  )^{W_\sigma}
 \longrightarrow C_0  (\mathfrak{a}^*_{P,+},\compop(\Ind H_\sigma )  )^{R_\sigma}
\]
given by restriction to the positive  chamber is an isomorphism.  The map is certainly injective, because if a $W_{0,\sigma}$-invariant function vanishes on the positive chamber $\mathfrak{a}^*_{P,+}$, then it must vanish on all chambers, and hence on $\mathfrak{a}^*_{P}$, since $W_{0,\sigma}$ acts transitively on the set of Weyl chambers.  So the issue is surjectivity.

Surjectivity follows from the fact that if $\nu \in \mathfrak{a}_P^*$,  if $w\in W_{0,\sigma}$, and if $w (\nu) = \nu$, then the  adjusted intertwining operator
\begin{equation*}
\mathscr{A}'_{w,\sigma, \nu}: \Ind H_{\sigma }\longrightarrow\Ind   H_{\sigma }
\end{equation*}
is a multiple of the identity operator; see for instance \cite[Theorem 1 (v)]{Delorme84}, or   the proof of Lemma 14.1 in \cite{KnappSteinII}. With this, we can extend any function $f\in C_0(\mathfrak{a}^*_{P,+},\compop(\Ind H_\sigma )  )$ to a $W_{0,\sigma}$-equivariant $C_0$-function on $\mathfrak{a}_P^*$ using  the formula
\[
f(w(\nu)) = \mathscr{A}'_{w,\sigma, \nu} f(\nu) \mathscr{A}'{}^*_{\!\!\!w,\sigma,\nu} \qquad  \forall w \in W_{0,\sigma} \,\, \forall \nu\in \mathfrak{a}^*_{P,+}.
\]

Because $R_\sigma$ normalizes $W_{0,\sigma}$, if the original function $f$ is $R_\sigma$-equi\-variant, then so is its extension.
\end{proof}

\subsection{Rescaling automorphisms and the R-group}
\label{sec-rescaling-without-w-prime-groups}
We are now ready to define our 
rescaling automorphisms 
\[
\alpha _t \colon  C_r^*(G) \longrightarrow  C_r^*(G) \qquad (t >0).
\]
We shall use the $C^*$-algebra isomorphism given in  Theorem~\ref{thm-structure-of-reduced-c=star-algebra-2}, and first define a family of rescaling automorphisms of each  
  summand in the $c_0$-direct sum that appears there.
Then we shall   transfer the automorphisms to $C^*_r(G)$ using the isomorphism in the theorem.

To begin, we shall disregard the $R$-groups that appear in Theorem~\ref{thm-structure-of-reduced-c=star-algebra-2} and construct, for any representative $(P,\sigma)$ of any associate class, a one-parameter group of automorphisms
\begin{equation}
    \label{eq-rescaling-without-r-group-invariants}
\alpha_t \colon C_0\bigl (\mathfrak{a}^*_{P},\compop(\Ind H_\sigma )\bigr ) 
\to C_0\bigl (\mathfrak{a}^*_{P},\compop(\Ind H_\sigma )\bigr ) 
\quad (t> 0)
\end{equation}
 with the property that 
\begin{equation}
    \label{eq-rescaling-representations-without-r-group-worries}
\pi_{P, \sigma, \nu}\circ \alpha_t \simeq \pi_{P, \sigma, t^{-1}\nu} \qquad \forall \nu \in \mathfrak{a}^*_P,\,\,\forall t>0
\end{equation}
(in the present context, we are using $\pi_{P, \sigma,\nu}$ to denote the representation of $C_0(\mathfrak{a}^*_{P},\compop(\Ind H_\sigma ) ) $ as compact operators on the Hilbert space $\Ind H_\sigma $ that is given by evaluation of functions at $\nu\in \mathfrak{a}^*_{P}$).  
Then we shall adjust this initial construction to account for the action of $R_\sigma$, and finally we shall show that we may restrict to $\mathfrak{a}^*_{P,+}\subseteq \mathfrak{a}^*_P$ so as to obtain a  one-parameter group of automorphisms
\begin{equation}
    \label{eq-rescaling-with-r-group-invariants}
\alpha_t \colon C_0\bigl (\mathfrak{a}^*_{P,+},\compop(\Ind H_\sigma )\bigr )^{R_\sigma} 
\to C_0\bigl (\mathfrak{a}^*_{P,+},\compop(\Ind H_\sigma )\bigr )^{R_\sigma}
\quad (t> 0)
\end{equation}
for which \eqref{eq-rescaling-representations-without-r-group-worries} still holds (for $\nu\in \mathfrak{a}^*_{P,+}$).

The initial construction is easy:    we  simply define 
\[
\alpha_t (f)(\nu) = f(t^{-1} \nu)\qquad \forall t > 0,\,\,\, \forall \nu \in \mathfrak{a}^*_{P}.
\]
But when $R_\sigma$ is nontrivial, the adjustment that must be made so as to preserve $R_\sigma$-equivariance is a quite bit more involved.
 
Recall the group $S_\sigma\subseteq N_K(\mathfrak{a}_P)/ Z_K(\mathfrak{a}_P)$ from Theorem~\ref{thm-Rgroup-sl2s}, which includes $R_\sigma$ as a subgroup.   It is generated by a finite commuting family  of simple reflections on $\mathfrak{a}_P^*$, and therefore we can make the following choice:

\begin{definition}\label{def-fundamental-chamber-F}
    Throughout this section we shall denote by   $F\subseteq \mathfrak{a}^*_{P}$   a fixed choice of  closed,  convex fundamental domain (an intersection of half-spaces) for the action of $S_\sigma $ on the Euclidean space $\mathfrak{a}^*_{P}$.
\end{definition}

Of course, the group  $S_\sigma$ acts simply-transitively on the set of all possible fundamental chambers, and the union of all possible fundamental chambers is the entirety of $\mathfrak{a}^*_{P}$.

\begin{definition}
\label{def-b-operator-and-alpha-w}
For  $w\in S_\sigma$, $\nu \in F$ and $t{>}0$, define a unitary operator 
\[
\mathscr{B}_{w,\nu,t} \colon \Ind H_\sigma \longrightarrow  \Ind H_\sigma
\]
by 
\[
\mathscr{B}_{w,\nu,t} =  \mathscr{A}^{\phantom{*}}_{w,w^{-1}(\sigma),\nu}
     \mathscr{A}^*_{w,w^{-1}(\sigma),t^{-1}\nu} .
\]
In addition, given $f\in C_0\bigl (\mathfrak{a}^*_{P},\compop(\Ind H_\sigma ) \bigr)$,  define a continuous function
    \[
    \alpha_t^w (f) \colon w (F)\longrightarrow \compop(\Ind H_\sigma ) ,
    \]
    vanishing at infinity, by means of the formula
    \[
     \alpha_t^w (f)\bigl (w(\nu)\bigr ) 
     = \mathscr{B}^{\phantom{*}}_{w,\nu,t}
     f\bigl (t^{-1} w(\nu)\bigr)
\mathscr{B}^{*}_{w,\nu,t}
    \]
    for all $\nu\in F$.
\end{definition}

We are going to show that as $w\in S_\sigma$ varies, the functions $\alpha^w_t(f)$ just defined agree on the pairwise intersections of  the closed subspaces $w(F)\subseteq \mathfrak{a}^*_{P}$, and so define, piecewise, a single continuous function
\[
    \alpha_t (f) \colon \mathfrak{a}^*_{P} \longrightarrow \compop(\Ind H_\sigma ) .
\]
This requires the following fact: 

\begin{lemma} 
\label{lem-u-mu-constant-in-some-directions}
Let $w\in S_\sigma$, and let $s\in S_\sigma$ be one of the generating simple reflections in part \textup{(ii)} of Theorem~\textup{\ref{thm-Rgroup-sl2s}}. If   $\nu \in \mathfrak{a}^*_{P}$, and if $\nu $ is fixed by $s$, then for every $t>0$  the unitary operator\footnote{Of course,  $s{=}s^{-1}$ since $S$ is a reflection, but we have included inverse signs in this and subsequent formulas nonetheless.} 
\[
\mathscr{A}^{\phantom{*}}_{s,s^{-1}w^{-1}(\sigma),\nu}
      \mathscr{A}^*_{s,s^{-1}w^{-1}(\sigma), t^{-1}\nu} \colon \Ind H_{w^{-1}(\sigma)} \longrightarrow \Ind H_{w^{-1}(\sigma)}
\]
is a scalar multiple of the identity.
\end{lemma}

\begin{proof}
Theorems~\ref{thm-intertwiners-constant-in-some-directions} and \ref{thm-Rgroup-sl2s} provide a strongly continuous family of unitary operators
\[
 \mathscr{A}_{\nu} \colon \Ind H_{s^{-1}w^{-1}(\sigma)}
 \longrightarrow \Ind H_{w^{-1}(\sigma)}
 \qquad (\nu \in \mathfrak{a}_P^*)
\]     
that intertwine the compact models of the parabolically induced representations $\pi_{P,s^{-1}w^{-1}(\sigma),\nu}$ and $\pi_{P, w^{-1}(\sigma),s(\nu)}$, and that have the property that 
\begin{equation}
    \label{eq-constancy-in-many-directions-recapped}
s(\mu) = \mu \quad \Rightarrow \quad \mathscr{A}_\nu = \mathscr{A}_{\mu+\nu}.
\end{equation}
Now, the unitary operators 
\[
\mathscr{A}_{s,s^{-1}w^{-1}(\sigma),\nu} \colon \Ind H_{s^{-1}w^{-1}(\sigma)}
 \longrightarrow \Ind H_{w^{-1}(\sigma)}
 \qquad (\nu \in \mathfrak{a}_P^*)
\]  
have the same intertwining property, and since the representations $\pi_{P,s^{-1}w^{-1}(\sigma),\nu}$ are irreducible for a dense set of values $\nu\in \mathfrak{a}_P^*$ \cite[Thm\,14.15]{KnappRepTheorySemisimpleGroups}, it follows from Schur's lemma that there is a continuous scalar  function $c(\nu)$  such that 
\begin{equation}
    \label{eq-two-interwiners-equal-up-to-scalar-factor}
 \mathscr{A}_{\nu} = c(\nu)\cdot \mathscr{A}_{s,s^{-1}w^{-1}(\sigma),\nu}  \qquad \forall \nu \in \mathfrak{a}_P^*.
\end{equation}
But  if $\nu\in \mathfrak{a}_P^*$, if $\nu$ is fixed by $s$, and if $t> 0$, then the element 
\[
\mu = \nu - t^{-1} \nu
\]
is also fixed by $s$.  So it follows from \eqref{eq-two-interwiners-equal-up-to-scalar-factor} and  \eqref{eq-constancy-in-many-directions-recapped} that 
\[
\mathscr{A}^{\phantom{*}}_{s,s^{-1}w^{-1}(\sigma),\nu}
      = \text{scalar} \cdot \mathscr{A}^{\phantom{*}}_{s,s^{-1}w^{-1}(\sigma), t^{-1}\nu},
\]
as required.
\end{proof}

\begin{theorem}
\label{thm-rescaling-autos-on-full-space-of-continuous-fns}
For every $t{>} 0$ there is a unique automorphism
\[
\alpha_t \colon C_0\bigl (\mathfrak{a}^*_{P},\compop(\Ind H_\sigma ) \bigr) \longrightarrow C_0\bigl (\mathfrak{a}^*_{P},\compop(\Ind H_\sigma ) \bigr)
\]
such that  
    \[
    \nu\in F,\,\,\, w\in S_\sigma  \quad \Rightarrow \quad 
     \alpha_t(f)(w(\nu)) = \alpha_t^w (f) (w(\nu)),
     \]
where $\alpha_t^w(f)$ is as presented in Definition~\textup{\ref{def-b-operator-and-alpha-w}}.
\end{theorem}

\begin{proof}
We need to show that  if $w,w'\in S_\sigma $ and $\nu,\nu' \in F$, then 
\[
 w(\nu) = w'(\nu')
 \quad \Rightarrow \quad 
 \alpha_t^{w} (f)\bigl (w(\nu) )
=\alpha_t^{w'} (f)\bigl (w'(\nu') ).
\]

To begin, because the closed subset $F \subseteq \mathfrak{a}^*_{P}$ is a fundamental domain for the action of $S_\sigma $, the condition $w(\nu){=}  w'(\nu')$ above implies that $\nu = \nu'$.

Furthermore  $w^{-1} w'$ is an element of the isotropy subgroup of $\nu$. Therefore $w^{-1} w'$ is a product of simple reflections $s_1,\dots , s_k$ in hyperplanes  that include $\nu$:
\[
w' = ws_1\dots s_k\quad \text{with $s_1(\nu) = s_2(\nu) = \cdots = s_k(\nu)=\nu$.}
\]
So by induction, it suffices to prove that for all  $w\in S_\sigma$ and $\nu\in F$,
\[
s(\nu)=\nu\quad \Rightarrow \quad 
\alpha_t^{w} (f)\bigl (w(\nu) )
=\alpha_t^{ws} (f)\bigl (w(\nu) ).
\]

When $s(\nu) {=} \nu$, the cocycle relation for intertwining operators \eqref{eq-cocycle-relation} asserts that
\begin{multline*}
\mathscr{A}_{ws, s^{-1}w^{-1} (\sigma) ,\nu} 
     =  \mathscr{A}_{ w,w^{-1}(\sigma),\nu} \mathscr{A}_{s,s^{-1}w^{-1}(\sigma),\nu}
     \\
     \text{and} \quad 
     \mathscr{A}_{ws,s^{-1}w^{-1}(\sigma),t^{-1}\nu} 
     =  \mathscr{A}_{ w,w^{-1}(\sigma) ,t^{-1}\nu} \mathscr{A}_{s,s^{-1}w^{-1}(\sigma), t^{-1}\nu},
\end{multline*}
and it follows from this and Lemma~\ref{lem-u-mu-constant-in-some-directions}  that 
\[
\begin{aligned}
\mathscr{B}_{ws,\nu,t} 
    & = 
\mathscr{A}^{\phantom{*}}_{ws,s^{-1}w^{-1}(\sigma),\nu}
     \mathscr{A}^*_{ws,s^{-1}w^{-1}(\sigma),t^{-1}\nu}
     \\
     & =  \mathscr{A}^{\phantom{*}}_{ w,w^{-1}(\sigma),\nu} \mathscr{A}^{\phantom{*}}_{s,s^{-1}w^{-1}(\sigma),\nu}
      \mathscr{A}^*_{s,s^{-1}w^{-1}(\sigma), t^{-1}\nu}
      \mathscr{A}^*_{ w,w^{-1}(\sigma) ,t^{-1}\nu} 
      \\
      & = 
      \text{scalar} \cdot \mathscr{A}^{\phantom{*}}_{ w,w^{-1}(\sigma),\nu}  
      \mathscr{A}^*_{ w,w^{-1}(\sigma) ,t^{-1}\nu} 
      \\
      & = \text{scalar} \cdot \mathscr{B}_{w,\nu,t}
\end{aligned}
\]
So it follows from Definition~\ref{def-b-operator-and-alpha-w} that 
\begin{multline*}
\alpha_t^{ws} (f)\bigl (w(\nu) )
     = \mathscr{B}^{\phantom{*}}_{ws,\nu,t}
     f\bigl (t^{-1} w(\nu)\bigr )\mathscr{B}^{*}_{ws,\nu,t}
     \\
     = \mathscr{B}^{\phantom{*}}_{w,\nu,t}
     f\bigl (t^{-1} w(\nu)\bigr )\mathscr{B}^{*}_{w,\nu,t}
     =
     \alpha_t^{w} (f)\bigl (w(\nu) ),
\end{multline*}
as required.
\end{proof}

We still need to check that $\alpha_t$ is $R_\sigma$-equivariant:
 
\begin{lemma}
\label{lem-r-equivariance-of-rescaling}
    If $f\in C_0 (\mathfrak{a}^*_{P},\compop(\Ind H_\sigma ) )$ and $r\in R_\sigma$, then 
    \[
    r(\alpha_t (f)) = \alpha_t (r(f)) \qquad \forall t > 0.
    \]

\end{lemma}

\begin{proof} 
If $r\in R_\sigma$, $w\in S_\sigma $ and $\nu \in F$, then 
by using the definition of the $R_\sigma$-action in \eqref{eq-w-sigma-action-on-c-star-algebra-1} and \eqref{eq-w-sigma-action-on-c-star-algebra-2}, we get 
\[
    E^*\cdot r(\alpha_t (f))(rw(\nu)) \cdot E
    = \mathscr{A}_{r,\sigma,w(\nu)}\cdot  \alpha_t (f)(w(\nu))\cdot  \mathscr{A}^*_{r,\sigma,w(\nu)},
\]
where the unitary operator $E \colon \Ind H_{r(\sigma)} \to \Ind H_{\sigma}$ comes from an equivalence of representations $r(\sigma) \simeq \sigma$, as in \eqref{eq-equivalence-of-reps-e}.  If we insert into the above formula the definition of $\alpha_t(f)$, then we obtain 
\[
 E^*\cdot r(\alpha_t (f))(rw(\nu))\cdot  E = U \cdot f(t^{-1} w(\nu)) 
   \cdot  U^* ,
\]
where  
\[
U = \mathscr{A}^{\phantom{*}}_{r,\sigma, w(\nu)} \mathscr{A}^{\phantom{*}}_{w,w^{-1}(\sigma),\nu} \mathscr{A}^{*}_{w,w^{-1}(\sigma),t^{-1}\nu}
\]
The formula for the unitary operator $U$  may be manipulated using the cocycle relation \eqref{eq-cocycle-relation} as follows:  
\[
\begin{aligned}
U 
    & = \mathscr{A}^{\phantom{*}}_{rw,w^{-1}(\sigma), \nu} \mathscr{A}^{*}_{w,w^{-1}(\sigma),t^{-1}\nu}
    \\
    & = \mathscr{A}^{\phantom{*}}_{rw,w^{-1}(\sigma), \nu} \mathscr{A}^{*}_{rw,w^{-1}(\sigma), t^{-1}\nu} \mathscr{A}_{r,\sigma,t^{-1}w(\nu)}
    \\
    & = \mathscr{A}^{\phantom{*}}_{rw,w^{-1}r^{-1}(\sigma'), \nu} \mathscr{A}^{*}_{rw,w^{-1}r^{-1}(\sigma'), t^{-1}\nu} \mathscr{A}_{r,\sigma,t^{-1}w(\nu)}
\end{aligned}
\]
where $\sigma' = r(\sigma)$ (this representation is equivalent to $\sigma$, but not equal to it).
But now, 
\[
    \mathscr{A}^{\phantom{*}}_{r,\sigma,t^{-1}w(\nu)} f(t^{-1} w(\nu)) \mathscr{A}^*_{r,\sigma,t^{-1}w(\nu)}    
 =   E^* \cdot (r(f)(t^{-1} rw(\nu))\cdot  E,
\]
which gives 
\begin{equation}
    \label{eq-r-equivariance-nearly-there}
E^*\cdot  r(\alpha_t (f))(rw(\nu))\cdot E  = W'\cdot  E^*\cdot  r(f)(t^{-1}rw(\nu))
\cdot E \cdot W'{}^{*},
\end{equation}
where 
\[
W' =  \mathscr{A}^{\phantom{*}}_{rw,w^{-1}r^{-1}(\sigma'), \nu} \mathscr{A}^{*}_{rw,w^{-1}r^{-1}(\sigma'), t^{-1}\nu}.
\]
If we set 
\[
W  = \mathscr{B}_{rw,\nu,t} =  \mathscr{A}^{\phantom{*}}_{rw,w^{-1}r^{-1}(\sigma), \nu} \mathscr{A}^{*}_{rw,w^{-1}r^{-1}(\sigma), t^{-1}\nu} ,
\]
then it follows from \eqref{eq-knapp-stein-a-and-e-relation} that  $EW'=WE$, and so \eqref{eq-r-equivariance-nearly-there} may be rewritten as 
\[
\begin{aligned}
E^*\cdot  r(\alpha_t (f))(rw(\nu))\cdot E  
    & = E^*\cdot W\cdot    r(f)(t^{-1}rw(\nu)) \cdot W{}^{*}\cdot E
    \\
    & =  E^*\cdot \alpha_t (r(f))(rw(\nu)) \cdot E ,
 \end{aligned}
\]
which proves the required $R_\sigma$-equivariance. 
\end{proof} 

With this, we have reached our first objective: 

\begin{theorem} 
\label{thm-rescaling-autos-on-r-invariants}
The automorphisms $\alpha_t$  in Theorem~\textup{\ref{thm-rescaling-autos-on-full-space-of-continuous-fns}} restrict to $R_\sigma$-invariants to give a continuous, one-parameter group of automorphisms
\[
\pushQED{\qed}
    \alpha_t \colon C_0\bigl (\mathfrak{a}^*_{P},\compop(\Ind H_\sigma ) \bigr)^{R_\sigma} \longrightarrow C_0\bigl (\mathfrak{a}^*_{P},\compop(\Ind H_\sigma ) \bigr)^{R_\sigma}\qquad (t>0).
\qedhere
\popQED
\]
\end{theorem}

To complete our construction of the rescaling automorphisms we require  just a small additional argument. 

\begin{lemma} 
There is a unique continuous, one-parameter group of automorphisms
\[
\alpha_t \colon C_0\bigl (\mathfrak{a}^*_{P,+},\compop(\Ind H_\sigma ) \bigr)^{R_\sigma} \longrightarrow C_0\bigl (\mathfrak{a}^*_{P,+},\compop(\Ind H_\sigma ) \bigr)^{R_\sigma}\qquad (t>0)
\]
such that for every $t>0$ the diagram 
\[
\xymatrix{
C_0\bigl (\mathfrak{a}^*_{P},\compop(\Ind H_\sigma ) \bigr)^{R_\sigma} \ar[r]^{\alpha_t}\ar[d]_{\mathrm{restr.}} & C_0\bigl (\mathfrak{a}^*_{P},\compop(\Ind H_\sigma ) \bigr)^{R_\sigma}\ar[d]^{\mathrm{restr.}}
\\
C_0\bigl (\mathfrak{a}^*_{P,+},\compop(\Ind H_\sigma ) \bigr)^{R_\sigma} \ar[r]_{\alpha_t} & C_0\bigl (\mathfrak{a}^*_{P,+},\compop(\Ind H_\sigma ) \bigr)^{R_\sigma}
}
\]
is commutative \textup{(}the automorphisms $\alpha_t$ at the top of the diagram are those of Theorem~\textup{\ref{thm-rescaling-autos-on-r-invariants}}, while the vertical arrows are given by restriction of functions from $\mathfrak{a}^*_{P}$ to $\mathfrak{a}^*_{P,+}$\textup{)}.
\end{lemma} 

\begin{proof} 
The formula for the automorphism $\alpha_t$ in the top row of the diagram above is 
\begin{equation}
\label{eq-formula-for-alpha-t-recap}
\alpha_t (f)(w(\nu)) =\mathscr{B}^{\phantom{*}}_{w,\nu,t}  f(t^{-1}w( \nu)) \mathscr{B}^*_{w,\nu,t} \qquad \forall w\in S_\sigma, \,\,\,\forall \nu \in F.
\end{equation}
It is evident from the formula that 
\[
f \vert_{\mathfrak{a}^*_{P,+}} = 0 \quad \Rightarrow \quad \alpha_t(f)  \vert_{\mathfrak{a}^*_{P,+}} = 0
\]
which means that the automorphism $\alpha_t$ in the top row maps the kernel of the restriction homomorphism in the diagram to itself, and therefore $\alpha_t$ passes to the quotient by this ideal. But we have seen in the proof of Theorem~\ref{thm-structure-of-reduced-c=star-algebra-2} that the restriction map is surjective, so the quotient is the $C^*$-algebra in the bottom row of the diagram.
\end{proof} 

To summarize, we have proved the following result: 

\begin{theorem}
    \label{thm-existence-of-rescaling-automorphisms}
    For every pair $(P,\sigma)$, the formula \eqref{eq-formula-for-alpha-t-recap}  defines a one-parameter group of $C^*$-algebra automorphisms 
    \[\alpha_t \colon C_0\bigl (\mathfrak{a}^*_{P,+},\compop(\Ind H_\sigma ) \bigr)^{R_\sigma}   \longrightarrow  C_0\bigl (\mathfrak{a}^*_{P,+},\compop(\Ind H_\sigma ) \bigr)^{R_\sigma}\qquad (t >0),
     \]
such that $\pi_{P,\sigma,\nu}\circ \alpha_t \simeq \pi_{P,\sigma,t^{-1}\nu}$ for all $\nu\in\mathfrak{a}_P^*$ and all $t>0$. \qed
\end{theorem}

\section{The limit formula}
\label{sec-limit-formula}

In this section and the next, we shall apply the rescaling automorphisms that were constructed in Section~\ref{sec-rescaling-automorphisms} to the study of the continuous field of $C^*$-algebras $\{ C^*_r (G_t)\}$ associated to the smooth family $\bigG$ from Section~\ref{sec-cts-fields}.  

For simplicity we shall work throughout with the restriction of the continuous field $\{ C^*_r (G_t)\}$ to the half-line $[0,\infty)$; the technique of \cite[Sec.\,4.2]{HigsonRoman20} could be used to extend to the whole line, but the essential features of the continuous field are already present in its restriction to the half-line.

The rescaling automorphisms constructed in Section~\ref{sec-rescaling-automorphisms} depend on our choices for representatives in each associate class $[P,\sigma]$ and our choices for fundamental domains for the actions of the commutative reflection groups  $S_\sigma$ on each $\mathfrak{a}_P$.  Different choices will lead to distinct but inner-equivalent one-parameter groups of automorphisms.  But  to avoid a treatment of these equivalences, we shall work with the fixed choices that we  made in Section~\ref{sec-rescaling-automorphisms}.

\subsection{Limit formula for matrix coefficients}
In the following formula, on the right-hand side, $\xi_t$ is to be viewed as an element of $C^*_r (G_t)$; recall from Section~\ref{sec-continuous-field-of-group-c-star-algebras} that as a group, $G_t$ identifies with $G$, but the $C^*$-algebras $C^*_r (G_t)$ are cons\-tructed using varying Haar measures.  In the formula, $\pi_{P,\sigma, t^{-1}\nu}$ is regarded as a representation of $C^*_r(G_t)$.

 \begin{theorem} 
 \label{thm-limit-formula-for-matrix-coefficients}
Let  $(P,\sigma)$ be a chosen representative of an associate class for a   real reductive group $G$, as above. Let   $\xi= \{ \xi_t\}$ be a continuous section of the continuous field  $\{ C^*_r (G_t )\}_{t\in [0,\infty)}$.    If $\nu\in \mathfrak{a}^*_{P}$, then
    \[
    \lim_{t\to 0}  \bigl \langle \varphi,  \pi_{P,\sigma,t^{-1}\nu} (\xi_t)\psi\bigr \rangle  = \bigl \langle \varphi, \rho_{\sigma\vert_{K\cap P},\nu} (\xi_0)\psi\bigr \rangle ,
    \]
for all $\varphi, \psi\in \Ind H_\sigma$, where the inner products are taken in $\Ind H_\sigma$ \textup{(}see Definition~\textup{\ref{def-rho-theta-nu}} and Remark~\textup{\ref{rem-extended-rho-notation}} for the definition of $\rho_{\sigma\vert_{K\cap P}, \nu}$\textup{)}.  The convergence is uniform in $\nu\in \mathfrak{a}^*_P$.
\end{theorem}

This is a generalization of Theorem~5.1.1 in \cite{HigsonRoman20}, which dealt  with the special case of minimal principal series representations of complex groups.   We begin with the following observation concerning  the deformation space $\bigG$, which is an immediate consequence of the definition of the smooth structure on $\bigG$, as described in Section~\ref{sec-dnc}.

\begin{lemma}
\label{lem-diffeomorphism-of-def-space-from-epsilon}
    If  $\varepsilon\colon G_0 \to G$ is a diffeomorphism that restricts to the identity map on the common subgroup $K$ of $G_0$ and $G$, then the function
    \[
    \begin{gathered}
    G_0\times \R\longrightarrow \bigG 
    \\
    (k,X,t)\longmapsto
    \begin{cases}
        ((k,X),0) & t = 0 
        \\
        (\varepsilon(k,tX),t) & t\ne 0
    \end{cases}
    \end{gathered}
    \]
    is a diffeomorphism. \qed
\end{lemma}

We shall use a diffeomorphism $\varepsilon\colon G_0\to G$ that is adapted to the parabolic subgroup $
P=L_PN_P  $ in the statement of Theorem~\ref{thm-limit-formula-for-matrix-coefficients}, as follows.
Fix  the Iwasawa decomposition
\begin{equation}
    \label{eq-iwasawa-decomp-of-l-p}
L_P= K_LAN_L
\end{equation}
 for which     
\[
K_L = K{\cap}L_P \quad \text{and} \quad N_L\subseteq N,
\]
and then define 
\begin{multline}
\label{eq-epsilon-diffeomorphism}
\varepsilon (k,[W{+}Y{+}Z]) = k\exp(W)\exp(Y)\exp(Z)
\\
\forall W\in \mathfrak{a}\,\,\,\forall Y\in \mathfrak{n}_L \,\,\, \forall Z \in \mathfrak{n}_P,
\end{multline}
where the square brackets indicate the class of $W{+}Y{+}Z$ in $\mathfrak{g}/\mathfrak{k}$.

For $t\ne  0$, define a \emph{rescaling diffeomorphism} $g\longmapsto g^t$ from $G$ to itself by
\begin{equation}
    \label{eq-def-of-self-diffeo}
\bigl (k \exp(W)\exp(Y)\exp (Z)\bigr )^t =  k \exp(tW)\exp(tY)\exp (tZ),
\end{equation}
with the same $W$, $Y$ and $Z$ as in \eqref{eq-epsilon-diffeomorphism}.  Note that this  restricts to a self-diffeomorphism of $P$, and that 
\[
\varepsilon (k,X)^t = \varepsilon (k,tX).
\]
So if for $g\in G$ we set  $g^0=\varepsilon^{-1}(g)\in G_0$, then Lemma~\ref{lem-diffeomorphism-of-def-space-from-epsilon} may be rewritten as follows:

\begin{lemma} 
\label{lem-big-g-up-to-diffeomorphism}
The rescaling map 
\[
\begin{gathered}
r \colon G\times \R \longrightarrow  \bigG   \\
r \colon (g,t)\longmapsto (g^t,t)
\end{gathered}
\]
is a diffeomorphism. \qed
\end{lemma}


We shall also need some standard formulas for invariant integration on reductive groups and parabolic subgroups.  

\begin{lemma}[See for example {\cite[Eqn\,(5)]{CowlingHaagerupHowe}}]
\label{lem-haar-measure-formulas}
Let  $P$ be a parabolic subgroup of a real reductive group $G$. If  $\delta_P \colon P\to \R^\times_+$ is the modular function for $P$ defined by \eqref{eq-def-of-delta-p} then the  formula 
\[
 \int_G \xi(g)\, dg = \int_K\int_P \xi(kp) \delta_P (p) \, dp\, dk,
\]
involving a left-invariant Haar integral on $P$, defines a Haar integral on the unimodular group $G$. Moreover the formula
\[
\int_P \xi(p)\, dp = \int_{L_P}\int_{N_P} \xi(\ell n)   \, dn\, d\ell
\]
defines a left-invariant Haar integral on $P$.
\end{lemma}

Lemma~\ref{lem-haar-measure-formulas} (applied twice, to $L_P$ and to $G$) yields the following change of variables  formula for the rescaling diffeomorphisms: 

\begin{lemma}
\label{lem-haar-measure-under-rescaling}
If $\xi$ is any continuous and compactly supported function on $G$, then 
    \begin{equation*}
   t^{-\operatorname{dim}(G/K)}  \int_G  \xi(g) \, dg 
   =  \int _K  \int_{P}    \xi(kp^t )
    \delta_P(p^t)\delta_Q(p^t)  \, dp \, dk .
    \end{equation*}
In the formula, the modular function $\delta_Q$ for the minimal parabolic subgroup $Q$ of $L$, is extended from $Q$ to a smooth function on $P$ by left $(K{\cap} L_P)$-invariance and right $N_P$-invariance.
\qed
\end{lemma}

\begin{proof}[Proof of Theorem~\ref{thm-limit-formula-for-matrix-coefficients}] 
The smooth and compactly supported functions  $\xi\colon \bigG\longrightarrow \C$ generate the continuous sections of $\{ C^*_r(G_t)\}_{t\in [0,\infty)}$ in the sense recalled in Section~\ref{sec-continuous-field-of-group-c-star-algebras}.    In addition, the representations $\pi_{P,\sigma,t^{-1}\nu}$ that appear in the statement of the theorem, are  $C^*$-algebra homomorphisms, and are therefore automatically norm-bounded by $1$. 
It follows  from  these facts and an $\varepsilon/3$-argument  that the general case of the theorem, involving an arbitrary continuous section $\xi=\{ \xi_t\}$, may be reduced to the special case in which $\xi$ is a smooth and compactly supported function on the deformation space $\bigG$. We shall consider only this  special case from now on.

In what follows, it will be convenient to view $\sigma$ as a representation of $L_P{=}M_PA_P$ that is trivial on $A_P$, and indeed as a representation of $P=L_PN_P$ that is trivial on $N_P$.  Similarly, it will be convenient to view $e^{i\nu}$ and similar 
as a character of $P=M_PA_PN_P$ that is trivial on $M_P$ and $N_P$.

For $t{>}0$ let us abbreviate  $\pi_{P,\sigma,t^{-1}\nu}$ to $\pi_{t^{-1}\nu}$, which we shall be viewing as a representation of $G_t$.  We are required to show  that if $\varphi, \psi$ are elements of the Hilbert space $\Ind H_\sigma = L^2(K, H_\sigma )^{K_L}$, then 
    \[
    \lim_{t\to 0}  \bigl \langle \varphi,  \pi_{t^{-1}\nu} (\xi_t)\psi\bigr \rangle_{\Ind H_\sigma}  = \bigl \langle \varphi, \rho_{\sigma\vert_{K \cap P} ,\nu} (\xi_0)\psi\bigr \rangle_{\Ind H_\sigma} ,
    \]
uniformly in $\nu$.  
In order to do so, we write 
\[
\begin{aligned}
\bigl \langle \varphi,  \pi_{t^{-1}\nu} (\xi_t)\psi\bigr \rangle
    & =
 t^{-\operatorname{dim}(G/K)}  \int_G \bigl \langle \varphi,  \pi_{t^{-1}\nu} (g)\psi\bigr  \rangle_{\Ind H_\sigma}\, \xi_t(g) \, dg     \\
    & =  t^{-\operatorname{dim}(G/K)} \int _K\int _G  \langle \varphi(k'),  \psi  (g^{-1}k') \rangle_{H_\sigma }
    \, \xi_t(g)\, dg\,dk' .
\end{aligned}
\]
Here we have used Fubini's theorem to reverse the order of integration. We have also extended $\psi$, initially an $H_\sigma $-valued function on $K$, to an $H_\sigma $-valued function on $G$ satisfying 
\begin{equation}
    \label{eq-covariance-relation}
\psi(kp) = p^{-(it^{-1}\nu +\rho)}\sigma (p)^{-1}\psi(k)
\qquad \forall k\in K\,\, \forall p\in P.
\end{equation}
(this is how the representation $\pi_{t^{-1}\nu}$ is defined in the compact picture;  $\rho$ is the half-sum of the positive restricted $\mathfrak{a}_P$-roots, so that $e^\rho$ is the square root of the modular function $\delta_P$ from \eqref{eq-def-of-delta-p}). Then we use the change of variables $g\to k'g^{-1}$ to write 
\begin{multline}
    \label{eq-first-inner-product-for-limit-formula}
\bigl \langle \varphi,  \pi_{t^{-1}\nu} (\xi_t)\psi\bigr \rangle
\\
=  t^{-\operatorname{dim}(G/K)} \int _K \int _G \langle \varphi(k'),   \psi(g) \rangle_{H_\sigma }
    \, \xi_t(k'g^{-1}) \, dg\,dk' .
\end{multline}
Now, writing $g=kp$,  formula \eqref{eq-covariance-relation} states that that 
\[
\langle \varphi(k'),\psi(g)\rangle_{H_\sigma } = \langle \varphi(k'),\psi(kp)\rangle_{H_\sigma } =   p^{-(it^{-1}\nu +\rho)}\langle \varphi(k'),\sigma (p)^{-1}\psi(k)\rangle_{H_\sigma }.
\]
With this, the integration formula in Lemma~\ref{lem-haar-measure-formulas}, and the change of variables formula in Lemma~\ref{lem-haar-measure-under-rescaling}, the right-hand side in \eqref{eq-first-inner-product-for-limit-formula} may be written  as
\begin{multline}
\label{eq-main-integral-for-limit-formula}
            \int_K \int _{P}     \int _K  p^{-(i\nu + t \rho)} \bigl \langle \varphi(k'),   \sigma (p^t)^{-1}\psi(k) \bigr \rangle_{H_\sigma } 
    \\
     \times  \xi_t\bigl (k'(kp^t)^{-1}\bigr ) \delta_Q(p^t)\delta_P (p^t)    \, dk' \,   dp \,dk .
\end{multline}
Here we have used in addition the formula $(p^t)^{it^{-1}\nu + \rho} =p^{i\nu + t\rho}$.

We are assuming that $\xi$ is a smooth and compactly supported function on the deformation space $\bigG$. The function 
\[
(k',p,k,t) \longmapsto \xi_t (k'(kp^t)^{-1} ) 
\]
in \eqref{eq-main-integral-for-limit-formula} can be written as a composition 
\[
\xymatrix{
K\times P \times K\times \R  \ar[r]^-{(*)}& 
K\times G \times \R \ar[r]^-{\mathrm{id}\times r}& 
K \times \bigG \ar[r]^-{(**)} &  \bigG  \ar[r]^-{\xi}& \C,
}
\]
where $r$ is the diffeomorphism  from Lemma~\ref{lem-big-g-up-to-diffeomorphism},  and ($*$) and ($**$) are the smooth and proper maps 
\[
(k',p,k,t)\longmapsto (k',kp,t) \quad \text{and} \quad 
(k,g,t) \longmapsto(kg^{-1},t),
\]  
respectively.  The composition is a smooth and compactly supported function, and therefore the integrand in \eqref{eq-main-integral-for-limit-formula} is smooth in all variables and uniformly compactly supported   on the Cartesian product $K{\times} P {\times} K$, as $t$ varies. The integrand therefore  converges as $t{\to} 0$, uniformly in $\nu\in \mathfrak{a}_P^*$, to the function 
\begin{equation*}
\begin{aligned}          
(k',p,k) 
& \longmapsto  
    a^{-i\nu } \bigl \langle \varphi(k'),    \sigma(k_L)^{-1}\psi( k) \bigr \rangle_{H_\sigma }
\xi_0 (k'(kp^0)^{-1})
            \\
& =   a^{-i\nu } \bigl \langle \varphi(k'),     \psi( k k_L) \bigr \rangle_{H_\sigma }
           \\
& \qquad  \qquad \times \xi_0\bigl (k'\cdot [{-}\log(a){-}\log(n_L){-}\log(n_P)]\cdot (kk_L)^{-1}\bigr ),
\end{aligned}
\end{equation*}
where $p=\ell_Pn_P$, where  $\ell_P = k_L a n_L$ in the Iwasawa decomposition of $L_P$, and where the dots $\cdot$  indicate multiplication in $G_0$. Compare \eqref{eq-epsilon-diffeomorphism}
The quantity $\langle \varphi,  \pi_{t^{-1}\nu} (\xi_t)\psi \rangle$ therefore converges  as $t{\to}0$,  uniformly in $\nu\in \mathfrak{a}_P^*$,  to
\begin{equation}
\label{eq-matrix-coeff-for-g-0}
\int _K\int _{\mathfrak{g}/\mathfrak{k}} \int_K
e^{-i \nu (X)} \bigl \langle \varphi(k'),     \psi( k) \bigr \rangle_{H_\sigma }
    \,  \xi_0 (k'\cdot (-X) \cdot k^{-1})
     dk'\,dX\,dk .
\end{equation}
But by a calculation similar to the one that led to \eqref{eq-main-integral-for-limit-formula},  the inner product $\langle \varphi, \rho_{\sigma\vert_{K \cap P} ,\nu}(\xi_0)\psi \rangle$ may be written as  \eqref{eq-matrix-coeff-for-g-0}, too.
\end{proof}

\subsection{Limit formula for  representations}

\begin{theorem} 
\label{thm-limit-formula-for-representations}
Let  $\xi= \{ \xi_t\}$ be a continuous section of the continuous field $\{ C^*_r (G_t )\}_{t\in [0,\infty)}$. Let $P=L_PN_P=M_PA_PN_P$ be a parabolic subgroup, and let $\sigma $ be a square-integrable irreducible unitary representation of $M_P$ and let $\nu\in \mathfrak{a}_P^*$.  The operators 
\[
\pi_{P,\sigma ,t^{-1}\nu} (\xi_t) \in \mathfrak{K} \bigl ( L^2(K,H_\sigma  )^{K{\cap}L_P}\bigr )\qquad (t>0)
\]
converge in norm to a limit, as $t$ tends to $0$, uniformly in $\nu\in \mathfrak{a}_P^*$.  The limit depends only on $\xi_0\in C^*(G_0)$ and in fact 
\[
\lim_{t\to 0} \pi_{P,\sigma ,t^{-1}\nu} (\xi_t) = \rho_{\sigma\vert_{K\cap P} ,\nu} (\xi_0).
\]
\end{theorem}

\begin{proof}
By an  $\varepsilon/3$-approximation argument, much as in the proof of Theorem~\ref{thm-limit-formula-for-matrix-coefficients}, it suffices to prove the result for all continuous, compactly supported and $K$-bi-finite functions $\xi$ on $\bigG $. 

For such $\xi$, the operators $\pi_{P,\sigma , t^{-1} \nu}(\xi_t)$ and $\rho_{\sigma\vert_{K \cap P} , \nu}(\xi_0)$ are all supported on a common finite-dimensional subspace of $L^2(K,H_\sigma )^{K\cap L_P}$ (in the sense that they and their adjoints are all zero on the orthogonal complement). If $\{ \varphi_{j}\}$ is an orthonormal basis for this finite-dimensional space, then the limit formula for operators stated in the theorem is equivalent to the finite family of limit formulas 
\[
    \lim_{t\to 0}  \bigl \langle \varphi_j,  \pi_{P,\sigma,t^{-1}\nu} (\xi_t)\varphi_k\bigr \rangle  = \bigl \langle \varphi_j, \rho_{\sigma\vert_{K \cap P},\nu} (\xi_0)\varphi_k\bigr \rangle \qquad \forall j,\,\,\,  \forall k 
\]
(with the limits uniform in $\nu$). So an appeal to Theorem~\ref{thm-limit-formula-for-matrix-coefficients} completes the proof. 
\end{proof}

\subsection{Limit formula for rescaling automorphisms}

\begin{definition}
\label{def-lambda-t-isomorphism}
For $t{\ne}0$ we shall denote by 
\begin{equation*}
\lambda_t \colon C^*_r (G_t)\stackrel \cong \longrightarrow C^*_r(G)
\end{equation*}
the $C^*$-algebra isomorphism such that 
\[
  C_c^\infty (G_t)\ni \xi_t \stackrel \cong \longmapsto  |t|^{-\dim (G/K)} \xi_t \in C_c^\infty (G)
\]
(the factor $|t|^{-\dim(G/K)}$  accounts for the change in Haar measures from $G_t$ to $G$).   
\end{definition} 

Here is the main result in Section~\ref{sec-limit-formula}: 

\begin{theorem}
    \label{thm-the-limit-exists-1}
    If  $\{ \xi_t\}_{t\ge 0}$ is any continuous section of the continuous field $\{ C^*_r (G_t)\}$ over $[0,\infty)$, then  the limit 
    \[
    \lim_{t\to 0} \alpha_t (\lambda_t( \xi_t)) 
    \]
    exists in $C^*_r(G)$.
\end{theorem}

The following lemma repeats the first step in our proof of Theorem~\ref{thm-limit-formula-for-matrix-coefficients}: 

\begin{lemma}{\cite[Lemma~5.1.3]{HigsonRoman20}} 
  \label{lem-reduce-to-generating-family}
  If the limit in Theorem~\textup{\ref{thm-the-limit-exists-1}} exists for a generating family of continuous sections of $\{C^*_r (G_t)\}$, then it exists for all continuous sections of $\{ C^*_r (G_t)\}$.
  \end{lemma}

Recall that if $K$ acts   continuously on a complex vector   space $W$, then a vector   $w{\in}W$ is said to be \emph{$K$-finite} if the linear span of  the orbit of $w$ under the action of $K$ is finite-dimensional and the action on this finite-dimensional space is continuous, or equivalently if $ w$   lies in the image under the natural map 
\[
\bigoplus_{\theta \in \widehat K}V_\theta \otimes _{\C}  \operatorname{Hom}_K(V_\theta, W) \longrightarrow W
\]
of the span of  finitely many summands $V_\theta \otimes _{\C}  \operatorname{Hom}_K(V_\theta, W)$ (here $V_\theta$ is the representation space for a representative of  $\theta \in \widehat K$).  We shall call the minimal set of $\theta{\in} \widehat K$ here the \emph{$K$-isotypical support} of $w {\in} W$.

\begin{lemma}{\cite[Lemma~5.1.4]{HigsonRoman20}}
  \label{lem-K-finite-generating-family1}
  There exists a generating family of continuous sections for the continuous field $\{ C^*_r (G_t)\}$ consisting of smooth and compactly supported functions on $\bigG $ that are $K$-finite for both the left and right translation actions of $K$ on $\bigG $.
  \end{lemma}

\begin{theorem}
\label{thm-generalized-uniform-admissability} Each  irreducible representation of $K$ occurs as  a $K$-type in only finitely many unitary equivalence classes of   principal series representations of the form $\pi_{\sigma,0}=\Ind_{P}^G \sigma\otimes 1 $, with $P{=}M_PA_PN_P$ a standard parabolic subgroup of $G$, and $\sigma$ a square-integrable irreducible unitary representation of $M_P$.
\end{theorem}

\begin{proof}
Fix a standard parabolic subgroup $P=M_PA_PN_P$.  Harish-Chandra proved in  \cite{HarishChandra66} (see \cite[\S~7.7]{Wallach1} for an exposition) that  each irreducible representation of $K{\cap} M_P$ occurs   in at most  finitely many mutually inequivalent square-integrable representations $\sigma$ of $M_P$.  By Frobenius reciprocity, an irreducible representation $\theta$ of $K$ occurs in $\pi_{\sigma,0}$ if and only if some irreducible constituent of $\theta\vert_{K\cap M_P}$ occurs in $\sigma$.  So $\theta$ occurs in at most finitely many of the representations $\pi_{\sigma,0}$.  The theorem follows from this because  there are only finitely many standard parabolic subgroups.
\end{proof}

\begin{corollary}[See {\cite[Lemma~5.1.5]{HigsonRoman20}}]
  \label{cor-K-finite-generating-family2}
Let $\{\xi_t\}$ be a  right $K$-finite   continuous section of  $\{ C^*_r (G_t)\}$. If for every associate class $[P,\sigma]$ the norm limit 
\[
\lim_{t\to 0} \pi_{P,\sigma} (\alpha_t (\lambda_t(\xi_t)))
\]
exists, then the norm limit 
$\lim_{t\to 0} \alpha_t (\lambda_t (\xi_t))$
exists in $C^*_r (G)$.
\end{corollary}

\begin{proof}
It suffices to show that a single fixed $K$-type $\theta$ is a $K$-type for only finitely many $\pi_{P,\sigma}$, which is guaranteed by the previous theorem. See the proof in \cite[Lemma~5.1.5]{HigsonRoman20} for details. 
\end{proof}

\begin{proof}[Proof of Theorem~\ref{thm-the-limit-exists-1}]
According to Lemma~\ref{lem-reduce-to-generating-family}, we only need to verify that the limit in the statement of the theorem exists for a generating family of continuous sections, and we shall use Lemma~\ref{lem-K-finite-generating-family1} to work with the generating family of continuous sections $\{ \xi_t\}$ associated to the   smooth, compactly supported, left and right $K$-finite functions on $\bigG$. Theorem~\ref{thm-limit-formula-for-representations} shows that for each associate class representative, the limit
\[
\lim_{t\rightarrow 0}  \pi_{P,\sigma}(\alpha_t(\xi_t))
\]
exists in $C_0(\mathfrak{a}_{P,+}^*,\compop(\Ind H_\sigma )) ^{W_\sigma}$. Corollary~\ref{cor-K-finite-generating-family2} completes the proof.
\end{proof}

The following formula for the limit in Theorem~\ref{thm-the-limit-exists-1} is an immediate consequence of Theorem~\ref{thm-limit-formula-for-representations} and the definition of the rescaling automorphisms: 

\begin{theorem}
    \label{thm-the-limit-exists-2}
    If  $\{ \xi_t\}_{t\ge 0}$ is any continuous section of the continuous field $\{ C^*_r (G_t)\}$ over $[0,\infty)$,   then for each chosen associate class representative $(P,\sigma)$ \textup{(}see the discussion at the beginning of this section\textup{)}, if $\nu{\in} F$,   if $w{\in} S_\sigma$, and if $w(\nu) {\in}   \mathfrak{a}^*_{P,+}$, then 
\[
    \pi_{P,\sigma,w(\nu)}\bigl ( \lim_{t\to 0} \alpha_t ( \lambda_t(\xi_t) )\bigr )  = \mathscr{A}'_{w,w^{-1}(\sigma),\nu}\cdot \rho_{w^{-1}(\sigma)\vert_{K\cap P},\nu} (\xi_0)\cdot \mathscr{A}'{}^{*}_{\!\!\!w,w^{-1}(\sigma),\nu}.
\]
    \qed
\end{theorem}

Thus (a unitary conjugate of) the representation $\rho_{\sigma\vert_{K \cap P}, \nu}$ of $G_0$ is a kind of limit of the representations $\pi_{P,\sigma,t^{-1}\nu}$ of $G$.  This relationship will be made precise in Lemma \ref{lem-composition-of-pi-with-alpha-1}.

\section{The main results}

\label{sec-bijection-characterization}

In this section we shall use the limit formula and the rescaling automorphisms constructed above to construct the Mackey embedding, following the strategy laid out in Lemma \ref{lem-embedding-from-limit-formula-and-mapping-cone-field}. In the three final subsections, we shall use the embedding to give simple characterizations of the  continuous field of $C^*$-algebras $\{ C^*_r (G_t)\}$, the Connes-Kasparov isomorphism,  and the Mackey bijection of Afgoustidis.

\subsection{Construction of the Mackey embedding}

Let $\xi_0\in C^*_r(G_0)$. Extend $\xi_0$  in any way to a continuous section $\{\xi_t\}$ of $\{C^*_r(G_t)\}$ and then form the limit
\begin{equation}
\label{eq-embedding-formula}
\alpha (\xi_0)= \lim_{t\rightarrow 0} \alpha_t (\lambda_t(\xi_t)) 
\end{equation}
in $C^*_r(G)$ using Theorem~\ref{thm-the-limit-exists-1}.

\begin{definition}
\label{def-mackey-embedding}
The Mackey embedding is the $C^*$-algebra homomorphism 
\[
\alpha \colon C^*_r(G_0)\longrightarrow C^*_r (G).
\]
determined by the formula \eqref{eq-embedding-formula} above.
\end{definition}

\begin{remarks}
If $\{ \xi'_t\}$ is a second extension of $\xi_0$ to a continuous section, then 
\[
\| \lim_{t\rightarrow 0}\alpha_t( \lambda_t (\xi_t)) -  \lim_{t\rightarrow 0}\alpha_t( \lambda_t (\xi'_t))\| 
=   \|  \lim_{t\rightarrow 0} \alpha_t( \lambda_t (\xi_t - \xi'_t)) \| 
= 0 .
\]
So the limit in \eqref{eq-embedding-formula} is independent of the extension of $\xi_0$ to a continuous section, and the Mackey embedding is well-defined.
Moreover, since the homomorphisms $\lambda_t$ and $\alpha_t$ are isometric, 
\[
\| \alpha (\xi_0)\| = \| \lim_{t\rightarrow 0} \alpha_t (\lambda_t(\xi_t)) \| 
= 
\lim_{t\rightarrow 0} \| \xi_t \|  = \| \xi_0\| ,
\]
and therefore the Mackey embedding is  an isometric embedding of $C^*$-algebras, as its name suggests.
\end{remarks}

\subsection{A characterization of the continuous field}  
\label{sec-mapping-cone}

We shall apply the mapping cone  construction from Definition~\ref{def-mapping-cone-field} to  the Mackey embedding  in Definition~\ref{def-mackey-embedding}.   

\begin{theorem} The $C^*$-algebra isomorphisms
\[
\begin{cases} 
\alpha_t \circ \lambda_t  \colon C^*_r (G_t) \longrightarrow C^*_r (G) & t \ne 0 \\
\;\;\; \operatorname{id} \colon  C^*_r (G_0) \longrightarrow C^*_r (G_0) & t = 0
\end{cases}
\]
define an isomorphism of continuous fields of $C^*$-algebras
from the continuous field 
$\{C^*_r(G_t)\}$ constructed from the smooth family $\{ G_t\}$  to  the mapping cone field for the embedding 
\[
\alpha \colon C^*_r(G_0)\longrightarrow C^*_r (G).
\]
\end{theorem}
  
\begin{proof} It suffices to show that for any continuous section $\{\xi_t\}$ of the continuous field $\{C^*_r(G_t)\}$, the image section of the mapping cone field is continuous; see  \cite[10.2.4]{Dixmier77}. But the  image section is $\{ \widehat \xi_t\}$, where 
\[
\widehat \xi_t =   
	\begin{cases} 
		\alpha_t(\lambda_t (\xi_t)) & t \ne 0 \\  
		\xi_0 & t=0 .
	\end{cases}
\]
This is obviously a continuous section of the mapping cone field away from $t{=}0$, and continuity at $t{=}0$ is proved using Theorem~\ref{thm-the-limit-exists-1} and the definition of $\alpha$.
\end{proof}
 
\subsection{The Connes-Kasparov isomorphism}
The Connes-Kasparov assembly map (first conjectured, and then proved to be an isomorphism)  was originally constructed using Dirac operators on the symmetric space $G/K$ and a Fredholm index  in $C^*$-algebra $K$-theory, as in  \cite{ConnesMoscovici82,Kasparov83}.  But Connes observed in \cite[Prop.9, p.141]{ConnesNCG} that the map may be identified with the bottom morphism in the following commuting diagram:
\begin{equation}
    \label{eq-c-k-morphism-from-continuous-field}
\xymatrix{
K_*(A_{[0,1]}) \ar[r]^{\varepsilon_1} \ar[d]_{\varepsilon _0}^{\cong} & K_*(A_1) \ar@{=}[d]
\\
K_*(A_0) \ar[r]_{\mathrm{CK}}&  K_*(A_1) .
}
\end{equation}
Here, $A_{[0,1]}$ denotes the $C^*$-algebra of continuous sections  over $[0,1]$ of the continuous field 
\[
\{\, A_t\, \} _{t \in \R} = \{\,C^*_r (G_t)\,\} _{t\in \R}
\]
of reduced group $C^*$-algebras  from Section~\ref{sec-cts-fields}.\footnote{At the same time, Connes pointed out how \eqref{eq-c-k-morphism-from-continuous-field} is connected to Mackey's idea of a correspondence between the irreducible unitary representations of $G$ and of $G_0$; this  played an  important role in reviving Mackey's proposal.}

The version of the Connes-Kasparov assembly map that appears in \eqref{eq-c-k-morphism-from-continuous-field} seems better suited to connections with the Mackey bijection and tempiric representations \cite{Higson08,AfgoustidisConnesKasparov,BraddHigsonYuncken24}.
In this section we shall develop Connes' observation a bit further by proving    the following result: 

\begin{theorem} 
\label{thm-connes-kasparov-characterization}
The Connes-Kasparov assembly map in \eqref{eq-c-k-morphism-from-continuous-field} is equal to the $K$-theory map 
\[
\alpha_{0,*}\colon K_*(A_0) \longrightarrow K_*(A_1)
\]
that is induced from the Mackey embedding $\alpha_0 \colon A_0 \longrightarrow A_1$.
\end{theorem}

For the proof, let  us write 
\[
\{ B_t \} _{t\in [0,1]} = \text{Mapping cone field for $\alpha_0\colon A_0 \longrightarrow A_1$}.
\]
As we have seen, there is an isomorphism of continuous fields 
\begin{equation}
    \label{eq-isomorphism-of-cts-fields-again}
\xymatrix{
\{ A_t \} _{t\in [0,1]} \ar[r]^{\cong}&  \{ B_t \} _{t\in [0,1]}
}
\end{equation}
given by the formula 
\begin{equation}
    \label{eq-isomorphism-of-cts-fields-again-fmla}
 a_t \longmapsto \begin{cases} \alpha_t(a_t) & t\ne 0 
\\
a_0 & t = 0 .
\end{cases}
\end{equation}
Notice that    $A_0 {=} B_0$ and $A_1 {=} B_1$, and that  on these fibers the isomorphisms in \eqref{eq-isomorphism-of-cts-fields-again-fmla} are  identity maps.

\begin{proof}[Proof of Theorem~\ref{thm-connes-kasparov-characterization}]
The diagram 
\[
\xymatrix{
B_{[0,1]} \ar[r]^{\varepsilon_1} \ar[d]_{\varepsilon _0} & B_1 \ar@{=}[d]
\\
B_0 \ar[r]_{\alpha_0}&  B_1
}
\]
is commutative up to homotopy, and therefore leads to an exactly commutative diagram
\begin{equation}
\label{eq-commuting-square-in-k-theory-for-b-field}
\xymatrix{
K_*(B_{[0,1]}) \ar[r]^{\varepsilon_{1,*}} \ar[d]_{\varepsilon _{0,*}} & K_*(B_1) \ar@{=}[d]
\\
K_*(B_0) \ar[r]_{\alpha_{0,*}}&  K_*(B_1) .
}
\end{equation}
Now, the isomorphism of continuous fields \eqref{eq-isomorphism-of-cts-fields-again} gives a commuting diagram
\begin{equation}
\label{eq-partial-commuting-cube}
\xymatrix@C=5pt@R=8pt{
& K_*(B_{[0,1]}) \ar[rr] \ar'[d][dd]
& & K_*(B_{1}) 
\\
K_*(A_{[0,1]}) \ar[ur] \ar[rr] \ar[dd] 
& &  K_*(A_{1}) \ar@{=}[ur] 
\\
& K_*(B_{0})  
& &  
\\
K_*(A_{0})  \ar@{=}[ur]
& &    
}
\end{equation}
in which the horizontal maps are induced from evaluation at $1$, while the vertical maps (which are isomorphisms) are induced from evaluation at $0$.  Putting \eqref{eq-partial-commuting-cube} together with \eqref{eq-c-k-morphism-from-continuous-field} and \eqref{eq-commuting-square-in-k-theory-for-b-field} we obtain the cubic diagram 
\[
\xymatrix@C=5pt@R=8pt{
& K_*(B_{[0,1]}) \ar[rr] \ar'[d][dd]
& & K_*(B_{1}) \ar@{=}[dd]
\\
K_*(A_{[0,1]}) \ar[ur]\ar[rr]\ar[dd]
& &  K_*(A_{1}) \ar@{=}[ur]\ar@{=}[dd]
\\
& K_*(B_{0}) \ar'[r]_-{\alpha_{0,*}}[rr]
& & K_*(B_{1})
\\
K_*(A_{0}) \ar[rr]_{\mathrm{CK}} \ar@{=}[ur]
& & K_*(A_{1}) \ar@{=}[ur]
}
\]
in which all faces commute. The theorem is proved.
\end{proof}

\subsection{A characterization of the Mackey bijection}
We conclude with our promised characterization of the Mackey bijection.

\begin{lemma} 
\label{lem-composition-of-pi-with-alpha-1}
Let $P=L N$ be a standard parabolic subgroup of $G$. Let $\nu \in \mathfrak{a}_P^*$, and assume that the centralizer of $\nu$ in $G$ is $L$.
Let $\tau$ be a tempiric representation of $L_P$,
and let
\[
\pi_{P,\tau,\nu} = \Ind_{P}^G e^{i\nu} \cdot \tau \colon G \longrightarrow U \bigl ( \Ind H_\tau \bigr ) 
\]
be the compact model of the irreducible representation that is unitarily parabolically induced from $e^{i\nu}\cdot  \tau$ \textup{(}notation from \eqref{eq-notation-e-to-the-i-nu-times-tau-1}; see Theorem~\textup{\ref{thm-afgoustidis-characterization-of-tempered-reps-with-fixed-im-inf-ch}} for the irreducibility of $\pi_{P,\tau,\nu}$\textup{)}. The composition
\[
C^*(G_0) \stackrel \alpha \longrightarrow C^*_r (G) \stackrel {\pi_{P,\tau,\nu}} \longrightarrow \mathfrak{K}(\Ind H_\tau)
\]
with the Mackey embedding in \eqref{eq-embedding-formula} is unitarily equivalent to the representation 
\[
\rho_{\tau\vert_{K\cap P},\nu} = \Ind_{(K{\cap}P)\ltimes \mathfrak{s}}^{K \ltimes \mathfrak{s}} \tau\vert_{K{\cap}L} \otimes e^{i \nu}
\]
of the group $G_0$.  On the right-hand side of the above formula,  the linear functional $\nu$ is extended  by zero on the orthogonal complement of $\mathfrak{a}_P$   to become a linear functional on $\mathfrak{s}$. 
\end{lemma}

\begin{proof} 
If $\tau$ is \emph{any} irreducible tempered unitary representation of $L_P$,  then there is a standard parabolic subgroup $Q{=}M_QA_QN_Q$ of $G$ with $Q\subseteq P$, a discrete series representation $\sigma$ of $M_Q$, some   $\mu\in \mathfrak{a}_Q^*$ and an embedding of representations  
\[
\tau\longrightarrow \Ind_{Q\cap L}^{L} \sigma \otimes \exp(i \mu)
\]
of the group $L$. This is a basic principle of Harish-Chandra, which is reflected in  Theorem~\ref{thm-structure-of-reduced-c=star-algebra-1}, when that theorem is applied to the real reductive group $L$ (note that the subgroups $Q {\cap} L\subseteq L$ are precisely the standard parabolic subgroups of $L$).
In our case,  since   $\tau$ has real infinitesimal character, the continuous parameter $\mu\in \mathfrak{a}_Q^*$ is $0$. Compare  \cite[(2.2.11)]{BraddHigsonYuncken24}.

By induction in stages, there is an embedding of unitary $G$-represen\-tations
\begin{equation}
    \label{eq-embedding-of-tau-into-principal-series}
\Ind_P^G e^{i\nu} {\cdot}  \tau \longrightarrow \Ind_Q^G  e^{i \nu}  {\cdot}\sigma .
\end{equation}
By Theorem~\ref{thm-limit-formula-for-representations}, the composition of the right-hand representation with the Mackey embedding is unitarily equivalent to the representation 
\[
\rho_{\sigma\vert_{K\cap Q},  \nu }\colon G _0 \longrightarrow U \bigl (L^2 (K, H_\sigma) ^{K\cap Q} \bigr ) .
\]
We need to compute the restriction of this representation to the image of the composition
\begin{equation}
    \label{eq-inclusion-of-ind-tau}
L^2 (K, H_\tau)^{K\cap P} \longrightarrow 
L^2 (K, L^2(K{\cap}P, H_\sigma)^{K\cap Q} )^{K\cap P}
\stackrel \cong  \longrightarrow L^2 (K, H_\sigma )^{K\cap Q} ,
\end{equation}
in which the first map is  induced from the embedding \eqref{eq-embedding-of-tau-into-principal-series} in the compact model, and the second is the induction in stages equivalence  \eqref{eq-induction-in-stages-formula}, given by evalution of functions in $L^2 (K{\cap}L, H_\sigma)^{K \cap Q}$ at $e\in H_\sigma$.  For this, we  use the following variation on induction in stages for the group $G_0$: if $K_2\subseteq K_1\subseteq K$, if $\sigma_2$ is a unitary representation of $K_2$,  if $\sigma_1= \Ind_{K_2}^{K_1} \sigma_2$, and if $\nu\in \mathfrak{s}^*$ is $K_1$-invariant, then the induction in stages unitary isomorphism 
\[
L^2 (K, L^2(K_1, H_{\sigma_2})^{K_2} )^{K_1}
\stackrel \cong  \longrightarrow L^2 (K, H_{\sigma_2} )^{K_2}
\]
is a unitary  equivalence of representations
    \[
    \Ind_{K_1\ltimes \mathfrak{s}}^{K\ltimes \mathfrak{s}} \sigma_1\otimes e^{ i \nu}
    \stackrel \cong \longrightarrow \Ind_{K_2\ltimes \mathfrak{s}}^{K\ltimes \mathfrak{s}}  \sigma_2\otimes e^{ i \nu}.
    \]
It follows that \eqref{eq-inclusion-of-ind-tau} is an inclusion of $\rho_{\tau\vert_{K\cap P},\nu} $ into $\rho_{\sigma\vert_{K\cap Q},\nu}$, as required.
\end{proof}

\begin{lemma} 
\label{lem-composition-of-pi-with-mackey-embedding-2}
Let $\pi$ be a tempered irreducible unitary representation of $G$. Fix $\nu\in \mathfrak{s} ^*$ so that   the imaginary part of the infinitesimal character of $\pi$ in the quotient 
\[
\mathfrak{a}^* / W(\mathfrak{g},\mathfrak{a} ) = \mathfrak{s}^*/K
\]
is represented by $\nu$.  The composition
\[
C^*(G_0) \stackrel \alpha \longrightarrow C^*_r (G) \stackrel \pi \longrightarrow \mathfrak{K}(H_\pi)
\]
is a direct sum of irreducible unitary representations   of $G_0$, all of which have the form $\rho_{\theta, \nu}$ from Definition~\ref{def-rho-theta-nu} for   $\theta\in \widehat K_\nu$, and for the given $\nu$.   Included among the direct summands is the irreducible representation of $G_0$ to which $\pi$ corresponds under the Mackey bijection of Afgoustidis, and this representation minimizes $\|\theta\|$ among all of the summands.
\end{lemma}

\begin{proof} 
If $\ImInfCh(\pi)\in \mathfrak{s}^*/K$ is represented by $\nu \in \mathfrak{s} 
^*$, then according to Theorem~\ref{thm-afgoustidis-characterization-of-tempered-reps-with-fixed-im-inf-ch}, $\pi$ may be taken to be a parabolically induced representation as in Lemma~\ref{lem-composition-of-pi-with-alpha-1} above.  According to that lemma, the composition of $\pi$ with the Mackey embedding is then the representation 
\[
 \Ind_{K_\nu\ltimes \mathfrak{s}}^{K \ltimes \mathfrak{s}} \tau\vert_{K{\cap} P} \otimes e^{i \nu}. 
\]
If we decompose $\tau\vert _{K{\cap} P}$ into irreducible subrepresentations, 
\[
\tau\vert _{K{\cap} P} = \theta_1 \oplus \theta_2 \oplus \cdots ,
\]
then there is a corresponding direct sum decomposition 
\[
\Ind_{K_\nu\ltimes \mathfrak{s}}^{K \ltimes \mathfrak{s}} \tau\vert_{K{\cap} P} \otimes e^{i \nu} = \rho_{\theta_1,\nu} \oplus 
\rho_{\theta_2,\nu}\oplus \cdots ,
\]
The lemma follows from this (for the assertions about the representation corresponding to $\pi$ under the Mackey bijection, recall that this is by definition $\rho_{\theta, \nu}$, where $\theta$ is the unique minimal $K$-type of $\tau$). 
\end{proof}

\begin{theorem}[\emph{c.f.} {\cite[Thm.~5.1.2]{HigsonRoman20}}]
\label{thm-characterization-of-mackey-bijection}
There is a unique bijection 
\[
\mu \colon \Gtemp  \longrightarrow \widehat G_0
\]
such that for every $\pi  \in \Gtemp$, the element $\mu(\pi )\in \widehat G_0$ may be realized as a unitary subrepresentation of the representation
\[
C^*(G_0) \stackrel \alpha \longrightarrow C^*_r (G) \stackrel \pi \longrightarrow \mathfrak{K}(H_\pi) .
\]
\end{theorem}

\begin{proof}
It was noted in Lemma~\ref{lem-composition-of-pi-with-mackey-embedding-2} that Afgoustidis's Mackey bijection (or rather, its inverse) has the property above, so it remains to prove uniqueness.   

It follows from Lemma~\ref{lem-composition-of-pi-with-mackey-embedding-2} any bijection with the property in the statement of the theorem must preserve the imaginary part of the infinitesimal character.  So, fixing   $\nu\in \mathfrak{a}^*$, defining $L$ to be the centralizer of $\nu$ in $G$, and $P$ the standard  parabolic subgroup with Levi factor $L_P{=}L$,  and taking into account Theorem~\ref{thm-afgoustidis-characterization-of-tempered-reps-with-fixed-im-inf-ch}, it must determine a bijection    
\[
\mu _\nu \colon 
  \widehat L_{\mathrm{tempiric}} 
\stackrel \cong \longrightarrow \widehat {K{\cap}L}
\]
defined by 
\[
\mu \colon \Ind_P^G e^{i\nu}\cdot \tau  \longmapsto \rho_{\mu_\nu(\tau) , \nu}.
\]
We need to show that the bijection $\mu_\nu$ maps $\tau$ to its unique minimal  $(K{\cap}L)$-type.

Now it follows from Lemma~\ref{lem-composition-of-pi-with-mackey-embedding-2} that the bijection $\mu_\nu $   has the property that 
\[
\| \mu_\nu    ( \tau  )\|  \ge \| \min ( \tau)\|  \qquad \forall\, \tau \in \widehat L_{\mathrm{tempiric}},
\]
where $\min ( \tau)$ is the unique minimal $K_\nu$-type of $\tau$, and so of course the inverse bijection 
has the property that 
\[
  \| \theta\|  \ge \| \min (\mu_\nu^{-1}   ( \theta   )) \|  
\qquad \forall\, \theta \in  \widehat {K{\cap}L}.
\]
That is, the composition of bijections 
\[
\xymatrix@C=30pt{
\widehat {K{\cap}L}   \ar[r]^-{\mu_\nu^{-1}} 
&\widehat L_{\mathrm{tempiric}}
\ar[r]^-{\min} &
\widehat {K{\cap}L}
}
\]
is norm-decreasing.  Since the set of $\theta$ with norm less than or equal to any given $C$ is finite,   it follows from this that  the composition above is actually norm preserving.  Since $\mu_\nu(\tau)$ must be a $(K{\cap}L)$-type of $\tau$, and since $\tau$ has a unique minimal $(K{\cap}L)$-type, it follows that $\mu_\nu(\tau)$ must be that minimal $(K{\cap}L)$-type, as it is for Afgoustidis's bijection.
\end{proof}

\begin{remark} 
Afgoustidis and Aubert proved in \cite{AfgoustidisAubert21} that the Mackey bijection is a continuous map 
in the direction $\mu\colon \Gtemp\longrightarrow \widehat G_0$.  This may be proved using Theorem~\ref{thm-characterization-of-mackey-bijection} and the additional fact that the Mackey bijection preserves minimal $K$-types \cite[Prop.\,4.1]{
AfgoustidisMackeyBijection}.
\end{remark}

\bibliographystyle{alpha}
\bibliography{biblio}

@misc{BraddHigsonYuncken24,
      title={Operator {K}-theory and tempiric representations}, 
      author={J. Bradd and N. Higson and R. Yuncken},
      year={2024},
      eprint={2412.18924},
        note={arXiv:2412.18924},
      archivePrefix={arXiv},
      primaryClass={math.RT},
      url={https://arxiv.org/abs/2412.18924}, 
}

@article {ConnesMoscovici82,
    AUTHOR = {Connes, A. and Moscovici, H.},
     TITLE = {The {$L^{2}$}-index theorem for homogeneous spaces of {L}ie
              groups},
   JOURNAL = {Ann. of Math. (2)},
  FJOURNAL = {Annals of Mathematics. Second Series},
    VOLUME = {115},
      YEAR = {1982},
    NUMBER = {2},
     PAGES = {291--330},
      ISSN = {0003-486X},
   MRCLASS = {58G12 (22E27 22E45 22E46 43A25 46L99 55R50)},
  MRNUMBER = {647808},
MRREVIEWER = {Jonathan M. Rosenberg},
       DOI = {10.2307/1971393},
       URL = {https://doi.org/10.2307/1971393},
}

@article {AfgoustidisAubert21,
    AUTHOR = {Afgoustidis, A. and Aubert, A.-M.},
     TITLE = {Continuity of the {M}ackey-{H}igson bijection},
   JOURNAL = {Pacific J. Math.},
  FJOURNAL = {Pacific Journal of Mathematics},
    VOLUME = {310},
      YEAR = {2021},
    NUMBER = {2},
     PAGES = {257--273},
      ISSN = {0030-8730},
   MRCLASS = {22E47 (22E50)},
  MRNUMBER = {4229239},
MRREVIEWER = {Efton Park},
       DOI = {10.2140/pjm.2021.310.257},
       URL = {https://doi.org/10.2140/pjm.2021.310.257},
}

@article {Delorme84,
    AUTHOR = {Delorme, P.},
     TITLE = {Homomorphismes de {H}arish-{C}handra li\'{e}s aux {$K$}-types
              minimaux des s\'{e}ries principales g\'{e}n\'{e}ralis\'{e}es des groupes de
              {L}ie r\'{e}ductifs connexes},
   JOURNAL = {Ann. Sci. \'{E}cole Norm. Sup. (4)},
  FJOURNAL = {Annales Scientifiques de l'\'{E}cole Normale Sup\'{e}rieure. Quatri\`eme
              S\'{e}rie},
    VOLUME = {17},
      YEAR = {1984},
    NUMBER = {1},
     PAGES = {117--156},
      ISSN = {0012-9593},
   MRCLASS = {22E45},
  MRNUMBER = {744070},
MRREVIEWER = {Mitsuo Sugiura},
       URL = {http://www.numdam.org/item?id=ASENS_1984_4_17_1_117_0},
}

@article {KnappSteinI,
    AUTHOR = {Knapp, A. W. and Stein, E. M.},
     TITLE = {Intertwining operators for semisimple groups},
   JOURNAL = {Ann. of Math. (2)},
  FJOURNAL = {Annals of Mathematics. Second Series},
    VOLUME = {93},
      YEAR = {1971},
     PAGES = {489--578},
      ISSN = {0003-486X},
   MRCLASS = {22E45},
  MRNUMBER = {460543},
MRREVIEWER = {G.\ I.\ Ol\cprime shanski\u i},
       DOI = {10.2307/1970887},
       URL = {https://doi.org/10.2307/1970887},
}

@article {KnappSteinII,
    AUTHOR = {Knapp, A. W. and Stein, E. M.},
     TITLE = {Intertwining operators for semisimple groups. {II}},
   JOURNAL = {Invent. Math.},
  FJOURNAL = {Inventiones Mathematicae},
    VOLUME = {60},
      YEAR = {1980},
    NUMBER = {1},
     PAGES = {9--84},
      ISSN = {0020-9910},
   MRCLASS = {22E46},
  MRNUMBER = {582703},
MRREVIEWER = {P. C. Trombi},
       DOI = {10.1007/BF01389898},
       URL = {https://doi.org/10.1007/BF01389898},
}

@book {KnappRepTheorySemisimpleGroups,
    AUTHOR = {Knapp, A. W.},
     TITLE = {Representation theory of semisimple groups},
    SERIES = {Princeton Mathematical Series},
    VOLUME = {36},
      NOTE = {An overview based on examples},
 PUBLISHER = {Princeton University Press, Princeton, NJ},
      YEAR = {1986},
     PAGES = {xviii+774},
      ISBN = {0-691-08401-7},
   MRCLASS = {22E46 (22-01 22E30)},
  MRNUMBER = {855239},
MRREVIEWER = {Wulf Rossmann},
       DOI = {10.1515/9781400883974},
       URL = {https://doi.org/10.1515/9781400883974},
}

@article {ClareHigsonSongTang24,
    AUTHOR = {Clare, P. and Higson, N. and Song, Y. and Tang,
              X.},
     TITLE = {On the {C}onnes-{K}asparov isomorphism, {I}},
   JOURNAL = {Jpn. J. Math.},
  FJOURNAL = {Japanese Journal of Mathematics},
    VOLUME = {19},
      YEAR = {2024},
    NUMBER = {1},
     PAGES = {67--109},
      ISSN = {0289-2316},
   MRCLASS = {22D25 (19K35 43A85 46L80)},
  MRNUMBER = {4717317},
       DOI = {10.1007/s11537-024-2220-2},
       URL = {https://doi.org/10.1007/s11537-024-2220-2},
}

@incollection {Vogan00,
    AUTHOR = {Vogan, Jr., D. A.},
     TITLE = {A {L}anglands classification for unitary representations},
 BOOKTITLE = {Analysis on homogeneous spaces and representation theory of
              {L}ie groups, {O}kayama--{K}yoto (1997)},
    SERIES = {Adv. Stud. Pure Math.},
    VOLUME = {26},
     PAGES = {299--324},
 PUBLISHER = {Math. Soc. Japan, Tokyo},
      YEAR = {2000},
   MRCLASS = {22E46 (22E47)},
  MRNUMBER = {1770725},
MRREVIEWER = {Edward G. Dunne},
       DOI = {10.2969/aspm/02610299},
       URL = {https://doi.org/10.2969/aspm/02610299},
}

@article {HigsonRoman20,
    AUTHOR = {Higson, N. and Rom\'{a}n, A.},
     TITLE = {The {M}ackey bijection for complex reductive groups and
              continuous fields of reduced group {C}*-algebras},
   JOURNAL = {Represent. Theory},
  FJOURNAL = {Representation Theory. An Electronic Journal of the American
              Mathematical Society},
    VOLUME = {24},
      YEAR = {2020},
     PAGES = {580--602},
   MRCLASS = {22E45 (46L99)},
  MRNUMBER = {4171564},
       DOI = {10.1090/ert/554},
       URL = {https://doi.org/10.1090/ert/554},
}

@book {Pedersen79,
    AUTHOR = {Pedersen, G. K.},
     TITLE = {{$C^{\ast} $}-algebras and their automorphism groups},
    SERIES = {London Mathematical Society Monographs},
    VOLUME = {14},
 PUBLISHER = {Academic Press, Inc. [Harcourt Brace Jovanovich, Publishers],
              London-New York},
      YEAR = {1979},
     PAGES = {ix+416},
      ISBN = {0-12-549450-5},
   MRCLASS = {46Lxx},
  MRNUMBER = {548006},
MRREVIEWER = {J. W. Bunce},
}

@article {AfgoustidisMackeyBijection,
    AUTHOR = {Afgoustidis, A.},
     TITLE = {On the analogy between real reductive groups and {C}artan
              motion groups: the {M}ackey-{H}igson bijection},
   JOURNAL = {Camb. J. Math.},
  FJOURNAL = {Cambridge Journal of Mathematics},
    VOLUME = {9},
      YEAR = {2021},
    NUMBER = {3},
     PAGES = {551--575},
      ISSN = {2168-0930},
   MRCLASS = {22E50 (22E46)},
  MRNUMBER = {4400734},
MRREVIEWER = {Jan Frahm},
       DOI = {10.4310/CJM.2021.v9.n3.a1},
       URL = {https://doi.org/10.4310/CJM.2021.v9.n3.a1},
}

@article {AfgoustidisConnesKasparov,
    AUTHOR = {Afgoustidis, A.},
     TITLE = {On the analogy between real reductive groups and {C}artan
              motion groups: a proof of the {C}onnes-{K}asparov isomorphism},
   JOURNAL = {J. Funct. Anal.},
  FJOURNAL = {Journal of Functional Analysis},
    VOLUME = {277},
      YEAR = {2019},
    NUMBER = {7},
     PAGES = {2237--2258},
      ISSN = {0022-1236},
   MRCLASS = {46L10 (22E30)},
  MRNUMBER = {3989145},
MRREVIEWER = {Efton Park},
       DOI = {10.1016/j.jfa.2019.02.023},
       URL = {https://doi.org/10.1016/j.jfa.2019.02.023},
}

@article {CowlingHaagerupHowe,
    AUTHOR = {Cowling, M. and Haagerup, U. and Howe, R.},
     TITLE = {Almost {$L^2$} matrix coefficients},
   JOURNAL = {J. Reine Angew. Math.},
  FJOURNAL = {Journal f\"ur die Reine und Angewandte Mathematik},
    VOLUME = {387},
      YEAR = {1988},
     PAGES = {97--110},
      ISSN = {0075-4102},
     CODEN = {JRMAA8},
   MRCLASS = {22D10 (22E46)},
  MRNUMBER = {946351 (89i:22008)},
MRREVIEWER = {Rebecca Herb},
}

@book{ConnesNCG, title = "Noncommutative {G}eometry", author = "A. Connes", year = "1994", publisher = "Academic Press", address = "San Diego"}

@book{T, author = "B. Bekka and P. de la Harpe and A. Valette", title = "Kazhdan's property ({T})", publisher = "Cambridge University Press", series = "New mathematical monographs", volume = "11", year = "2008"}

@article {CHS,
    AUTHOR = {Clare, P. and Higson, N. and Song, Y.},
     TITLE = {On the {C}onnes-{K}asparov isomorphism, {II}},
   JOURNAL = {Jpn. J. Math.},
  FJOURNAL = {Japanese Journal of Mathematics},
    VOLUME = {19},
      YEAR = {2024},
    NUMBER = {1},
     PAGES = {111--141},
      ISSN = {0289-2316,1861-3624},
   MRCLASS = {22D25 (22E45 22E47 46L80)},
  MRNUMBER = {4717318},
       DOI = {10.1007/s11537-024-2221-1},
       URL = {https://doi.org/10.1007/s11537-024-2221-1},
}

@article {HarishChandra66,
    AUTHOR = {Harish-Chandra},
     TITLE = {Discrete series for semisimple {L}ie groups. {II}. {E}xplicit
              determination of the characters},
   JOURNAL = {Acta Math.},
  FJOURNAL = {Acta Mathematica},
    VOLUME = {116},
      YEAR = {1966},
     PAGES = {1--111},
      ISSN = {0001-5962},
   MRCLASS = {22.65},
  MRNUMBER = {219666},
MRREVIEWER = {A. Borel},
       DOI = {10.1007/BF02392813},
       URL = {https://doi.org/10.1007/BF02392813},
}

@incollection {Higson08,
    AUTHOR = {Higson, N.},
     TITLE = {The {M}ackey analogy and {$K$}-theory},
   BOOKTITLE = {Group Representations, Ergodic Theory, and Mathematical Physics: A Tribute to {G}eorge {W}. {M}ackey},
   SERIES = {Contemp. Math.},
   VOLUME = {449},
   EDITORS = {{R}. {S}. {D}oran,   {C}. {C}. Moore,   and {R}. {J}. {Z}immer},
   PUBLISHER = {Amer. Math. Soc.}, 
    ADDRESS = {Providence, RI},
   YEAR={2008},
   PAGES={149--172},
}

@article {Kasparov83,
    AUTHOR = {Kasparov, G. G.},
     TITLE = {Index for invariant elliptic operators, {$K$}-theory and
              representations of {L}ie groups},
   JOURNAL = {Soviet Math. Dokl.},
    VOLUME = {27},
      YEAR = {1983},
    NUMBER = {1},
     PAGES = {105--109},
  MRNUMBER = {691088},
MRREVIEWER = {Jan Janas},
}

@book {VoganGreenBook,
    AUTHOR = {Vogan, Jr., D. A.},
     TITLE = {Representations of real reductive {L}ie groups},
    SERIES = {Progress in Mathematics},
    VOLUME = {15},
 PUBLISHER = {Birkh\"auser, Boston, Mass.},
      YEAR = {1981},
     PAGES = {xvii+754},
      ISBN = {3-7643-3037-6},
   MRCLASS = {22E47 (22E46)},
  MRNUMBER = {632407 (83c:22022)},
MRREVIEWER = {Joe Repka},
}

@incollection {Vogan07,
    AUTHOR = {Vogan, Jr., D. A.},
     TITLE = {Branching to a maximal compact subgroup},
 BOOKTITLE = {Harmonic analysis, group representations, automorphic forms
              and invariant theory},
    SERIES = {Lect. Notes Ser. Inst. Math. Sci. Natl. Univ. Singap.},
    VOLUME = {12},
     PAGES = {321--401},
 PUBLISHER = {World Sci. Publ., Hackensack, NJ},
      YEAR = {2007},
   MRCLASS = {22E47 (20G05)},
  MRNUMBER = {2401817},
MRREVIEWER = {Anne-Marie H. Aubert},
       DOI = {10.1142/9789812770790\_0010},
       URL = {https://doi.org/10.1142/9789812770790_0010},
}

@book {Wallach1,
    AUTHOR = {Wallach, N. R.},
     TITLE = {Real reductive groups. {I}},
    SERIES = {Pure and Applied Mathematics},
    VOLUME = {132},
 PUBLISHER = {Academic Press, Inc., Boston, MA},
      YEAR = {1988},
     PAGES = {xx+412},
      ISBN = {0-12-732960-9},
   MRCLASS = {22E46 (17B10 22-02 22E30)},
  MRNUMBER = {929683 (89i:22029)},
MRREVIEWER = {Roberto J. Miatello},
}

@article {Mackey49,
   AUTHOR = {Mackey, G. W.},
   TITLE = {Imprimitivity for representations of locally compact groups.
             {I}},
 JOURNAL = {Proc. Nat. Acad. Sci. U. S. A.},
 FJOURNAL = {Proceedings of the National Academy of Sciences of the United
             States of America},
   VOLUME = {35},
     YEAR = {1949},
   PAGES = {537--545},
     ISSN = {0027-8424},
 MRCLASS = {20.0X},
 MRNUMBER = {0031489},
MRREVIEWER = {R. Godement},
     DOI = {10.1073/pnas.35.9.537},
     URL = {https://doi.org/10.1073/pnas.35.9.537},
}

@preamble{
   "\def\cprime{$'$} "
}

@article {CCH16,
    AUTHOR = {Clare, P. and Crisp, T. and Higson, N.},
     TITLE = {Parabolic induction and restriction via {$C^*$}-algebras and
              {H}ilbert {$C^*$}-modules},
   JOURNAL = {Compos. Math.},
  FJOURNAL = {Compositio Mathematica},
    VOLUME = {152},
      YEAR = {2016},
    NUMBER = {6},
     PAGES = {1286--1318},
      ISSN = {0010-437X},
   MRCLASS = {22E46 (22D25 46L08)},
  MRNUMBER = {3518312},
MRREVIEWER = {Alexander Isaakovich Shtern},
       DOI = {10.1112/S0010437X15007824},
       URL = {https://doi.org/10.1112/S0010437X15007824},
}

@book {Dixmier77,
    AUTHOR = {Dixmier, J.},
     TITLE = {{$C\sp*$}-algebras},
      NOTE = {Translated from the French by Francis Jellett,
              North-Holland Mathematical Library, Vol. 15},
 PUBLISHER = {North-Holland Publishing Co., Amsterdam-New York-Oxford},
      YEAR = {1977},
     PAGES = {xiii+492},
      ISBN = {0-7204-0762-1},
   MRCLASS = {46L05},
  MRNUMBER = {0458185},
}

@book {Knapp02,
    AUTHOR = {Knapp, A. W.},
     TITLE = {Lie groups beyond an introduction},
    SERIES = {Progress in Mathematics},
    VOLUME = {140},
   EDITION = {Second},
 PUBLISHER = {Birkh\"auser Boston Inc.},
   ADDRESS = {Boston, MA},
      YEAR = {2002},
     PAGES = {xviii+812},
      ISBN = {0-8176-4259-5},
   MRCLASS = {22-01},
  MRNUMBER = {MR1920389 (2003c:22001)},
}

@incollection {Mackey75,
    AUTHOR = {Mackey, G. W.},
     TITLE = {On the analogy between semisimple {L}ie groups and certain related semi-direct product groups},
 BOOKTITLE = {Lie groups and their representations (Proc. Summer School,
              Bolyai J\'anos Math. Soc., Budapest, 1971)},
     PAGES = {339--363},
 PUBLISHER = {Halsted, New York},
      YEAR = {1975},
   MRCLASS = {22E45 (22D30)},
  MRNUMBER = {53 \#13478},
MRREVIEWER = {A. Kleppner},
}

@book {Fulton84,
    AUTHOR = {Fulton, W.},
     TITLE = {Intersection theory},
    SERIES = {Ergebnisse der Mathematik und ihrer Grenzgebiete (3) [Results
              in Mathematics and Related Areas (3)]},
    VOLUME = {2},
 PUBLISHER = {Springer-Verlag, Berlin},
      YEAR = {1984},
     PAGES = {xi+470},
      ISBN = {3-540-12176-5},
   MRCLASS = {14C17 (14-02 14C40)},
  MRNUMBER = {732620},
MRREVIEWER = {Werner Kleinert},
       DOI = {10.1007/978-3-662-02421-8},
       URL = {https://doi.org/10.1007/978-3-662-02421-8},
}

\end{document}